\documentclass[a4paper,10pt]{amsart}
 \usepackage{amsfonts,mathrsfs}
 \usepackage{amsthm}
 \usepackage{amssymb}
 \usepackage{hyperref}
 \usepackage{cleveref}
 \usepackage{subcaption}
 \usepackage{graphicx}
 \usepackage{enumerate}
\usepackage{tikz-cd}
\usepackage{tikz}
\usepackage{pgfplots}
\pgfplotsset{compat=1.15}
\usepackage{mathrsfs}
\usepackage{fancyhdr}
\usepackage{centernot}
\usetikzlibrary{arrows,backgrounds}
\usepackage{amsmath,amssymb,latexsym,graphicx}
\usepackage{keyval,ifthen}
\usepackage{xparse}
\usepackage{xcolor}
\usepackage{xargs}
\usepackage{etoolbox}
\usepackage{pgffor}
\usepackage{xstring}
\usepackage{xspace}
\usepackage{longtable}
\usetikzlibrary{shapes.geometric}
\usetikzlibrary{calc}
\usepackage{hyperref}
\usepackage{cleveref}
\usepackage[textsize=footnotesize,backgroundcolor=orange]{todonotes}
\usepackage{mathrsfs}
\usepackage{pythonhighlight}
\usepackage{listings}
\theoremstyle{definition}
\newtheorem{Def}{Definition}[section]
\newtheorem{ex}[Def]{Example}
\newtheorem{rem}[Def]{Remark}

\theoremstyle{plain}
\newtheorem{prop}[Def]{Proposition}
\newtheorem{thm}[Def]{Theorem}
\newtheorem*{thm*}{Theorem}
\newtheorem{lem}[Def]{Lemma}
\newtheorem{cor}[Def]{Corollary}
\newtheorem*{cor*}{Corollary}

\newtheorem*{con*}{Conjecture}
\newtheorem*{frag*}{Question}
\newtheorem*{verm*}{Vermutung}

\newcommand{\RR}{\mathbb{R}}
\newcommand{\CC}{\mathbb{C}}
\newcommand{\ZZ}{\mathbb{Z}}
\newcommand{\NN}{\mathbb{N}}
\newcommand{\QQ}{\mathbb{Q}}

\renewcommand{\d}{\partial}
\newcommand{\rank}{\operatorname{rk}}

\newcommand{\conv}{\operatorname{conv}}
\newcommand{\Ext}{\operatorname{Ext}}
\newcommand{\CoExt}{\operatorname{CoExt}}
\newcommand{\im}{\operatorname{Im}}

\newcommand{\newt}{\operatorname{Newt}}
\newcommand{\tr}{\operatorname{tr}}

\title[Matroids on Eight Elements with HPP]{Matroids on Eight Elements with the Half-plane Property and Related Concepts}
\author{Mario Kummer}
\address{Technische Universit\"at Dresden, Germany} 
\email{mario.kummer@tu-dresden.de}

\author{B\"u\c{s}ra Sert}
\address{DLR-German Aerospace Center, Germany} 
\email{buesra.sert@dlr.de}

\thanks{Both authors have been supported by the DFG under Grant No.421473641.}
\newcommand{\comment}[1]{}

\begin{document}


\begin{abstract}
 We classify all matroids with at most $8$ elements that have the half-plane property,  and we provide a list of some matroids on $9$ elements that have, and that do not have the half-plane property. Furthermore, we prove that several classes of matroids and polynomials that are motivated by the theory of semidefinite programming are closed under taking minors and under passing to faces of the Newton polytope.
\end{abstract}
\maketitle

\section{Introduction}
A multivariate polynomial is said to have the \emph{half-plane property} (HPP) if it does not vanish whenever we plug in for each variable a complex number from the open right half-plane. One of the initial motivations for this property comes from electrical networks; the theory suggests that the spanning tree polynomial of any graph should have the half-plane property by Kirchhoff's matrix tree theorem. Later on, in their seminal paper Choe, Oxley, Sokal and Wagner  \cite{Choe_2004} generalized this setup to arbitrary, not necessarily graphic, matroids by showing that the support of every homogeneous multiaffine polynomial with the half-plane property must be the collection of bases of a matroid. Conversely, they define a matroid to have the half-plane property if its basis generating polynomial has the half-plane property, and they initiate a thorough investigation of which matroids have the half-plane property. These matroids generalize graphic matroids, and they have several combinatorially interesting properties. In particular, they share a plethora of so-called negative dependence properties that are not satisfied by arbitrary matroids, see e.g. \cite{negdep,ineqs}. Moreover, non-negativity of the Rayleigh difference
\[
  \Delta_{ij}(h):= \frac{\d{h}}{\d{x_{i}}}\cdot\frac{\d{h}}{\d{x_{j}}}
  -\frac{\d^{2}{h}}{\d{x_{i}}\d{x_{j}}}\cdot h
\]
of the basis generating polynomial of such matroids for all indices $i,j$ is one of their characteristics. The class of matroids with the half-plane property has several desirable properties such as being closed under minors, duality,
direct sums, 2-sums \cite{Choe_2004}. Furthermore, it includes several families of matroids like uniform, regular and $\sqrt[6]{1}$-matroids. On the other hand, it is rather difficult to determine whether a given matroid has the half-plane property. Oxley, for instance, asks for a ``practically feasible algorithm for testing whether a
matroid is HPP'' in his book \cite[Problem 15.8.10]{oxley}. One of the main results of this paper is a complete classification of which matroids on at most $8$ elements have the half-plane property. This is achieved by checking theoretical criteria on the Rayleigh differences due to Br\"and\'en \cite{Branden2007} and Wagner--Wei \cite{Wagner_2009} by computational means using state-of-the-art software packages like the packages ``SumsOfSquares'' \cite{SumsOfSquaresArticle} and ``Matroids'' \cite{MatroidsArticle} for Macaulay2 \cite{M2} and the Julia package ``HomotopyContinuation.jl'' \cite{inbook}. Even though these methods are mostly numerical, our heuristical algorithm produces for each matroid with at most $8$ elements a \emph{symbolic} certificate for whether it has the half-plane property or not. Our classification yields a list of $32$ matroids with at most $8$ elements that are minor-minimal with respect to not having the half-plane property. These include the ten forbidden minors of rank $3$ on $7$ elements that were already found in \cite{Wagner_2009}, namely the Fano matroid, three of its relaxations, the free extension of $M(K_4)$ by one element and their duals. All other forbidden minors are sparse paving matroids of rank $4$ on $8$ elements. These include the self-dual matroid $P_8$ and three of its relaxations as well as the co-extension of $P_7$ and its dual.  Moreover,  we conduct tests on the half-plane property on matroids on $9$ elements of rank $3$ and $4$. Our tests confirm that the Pappus, non-Pappus and (non-Pappus$\backslash 9$)$+e$ matroids are the only forbidden minors for the half-plane property that are on $9$ elements with rank $3$. Among those of rank $4$, we provide a list of $4125$ matroids that have the half-plane property, and a list of $1218$ matroids that are forbidden (minimal) minors for the half-plane property. 

There is also ample interest in the half-plane property from the point of view of convex optimization. Consider for instance, the feasible sets of \emph{semidefinite programming} (SDP) that are so-called \emph{spectrahedral cones}. These are semi-algebraic convex sets that can (up to a linear change of coordinates) be described as the closure of the connected component of $C_f=\{x\in\RR^n:\,f(x)\neq0\}$ that contains the positive orthant for a suitable homogeneous polynomial $f$ with the half-plane property. The question whether, conversely, the set $C_f$ is a spectrahedral cone for every homogeneous $f$ with the half-plane property is the content of the \emph{generalized Lax conjecture}, a fundamental open question in this context \cite{Vinnikov_2012}. In \cite{Branden2011} Br\"and\'en showed that the basis generating polynomial of the V\'amos matroid $V_8$ constitutes a counter-example to some stronger predecessor of the generalized Lax conjecture. Further such examples were given in \cite{burton2014real, Amini2018} and will be given in \Cref{sosdet} of this article (on the other hand, see \cite{KumBez} for some positive result). This indicates that basis generating polynomials of matroids might be a good testing ground for the (current version of the) generalized Lax conjecture as well. We say that a matroid is \emph{spectrahedral} if its basis generating polynomial $f$ has the half-plane property and if the set $C_f$ defined as above is a spectrahedral cone. For example it is not known whether the V\'amos matroid is spectrahedral, although this is true for a related polynomial \cite{kumvam}. In \Cref{sec:hypspec} we will show that the class of spectrahedral matroids is minor-closed. More generally, we show in \Cref{sec1} that given $f$ such that $C_f$ is spectrahedral, the same is true for the polynomial that constitutes of all terms of $f$ whose monomial lies in a fixed face of the Newton polytope of $f$.

\section{The half-plane property and related concepts}\label{C}
Let $M$ be a matroid of rank $r$ on the ground set $E=\left[ n \right]$ and let $\mathcal{B}$ be the collection of its bases. The \emph{basis generating polynomial} of $M$ is defined to be $$h_M=\sum_{B \in \mathcal{B}}\prod_{i \in B} x_i.$$
Note that the basis generating polynomial of a matroid is \emph{multiaffine} in the sense that it has degree at most one in every variable.
We denote by $\rank: 2^{E}\rightarrow \NN$ the rank function of $M$. 
See \cite{oxley} for an introduction to matroid theory.

\begin{Def}
 A polynomial $0\neq f\in\CC[x_1,\ldots,x_n]$ is said to have the \emph{half-plane property} if $f(x_{1},\dots,x_{n})\neq 0$ whenever the real part of each $x_i\in\CC$ is positive.
\end{Def}


\begin{rem}
 For homogeneous polynomials, having the half-plane property is equivalent to being \emph{stable}. A polynomial $0\neq f\in\CC[x_1,\ldots,x_n]$ is stable if $$f(x_1,\ldots,x_n)\neq0$$ whenever $\im(x_i)>0$ for all $i\in[n]$. A univariate \emph{real} polynomial $f\in\RR[t]$ is stable if and only if it has only real roots.
\end{rem}

\begin{Def}
 A matroid is said to have the \emph{half-plane property} if its basis generating polynomial is stable.
\end{Def}

\begin{rem}
 If $ f,g\in\CC[x_1,\ldots,x_n]$ are both non-zero, then the product $f\cdot g$ is stable if and only if both $f$ and $g$ are stable. Since for matroids $M_1$ and $M_2$ we have $h_{M_1\oplus M_2}=h_{M_1}\cdot h_{M_2}$ \cite[Proposition~4.5]{Choe_2004}, we can often restrict our attention to connected matroids when studying the half-plane property.
\end{rem}

\begin{Def}
 Consider a polynomial $f\in\CC[x_1,\ldots,x_n]$ with support $\mathcal{S}\subseteq \ZZ_{\geq 0}^{n}$ given as $$f=\sum_{\alpha\in\mathcal{S}}c_\alpha x^{\alpha}$$ for $c_\alpha \in \CC\setminus\{0\}$. The \emph{Newton polytope} $\newt(f)$ of $f$ is the convex hull in $\RR^n$ of all $\alpha$ with $c_\alpha\neq0$. Let $F$ be a face of $\newt(f)$. We denote $$f_F=\sum_{\alpha\in\mathcal{S}\cap F}c_\alpha x^{\alpha}.$$
\end{Def}

A standard argument shows that passing to a face of the Newton polytope preserves the half-plane property.

\begin{prop}\label{prop:facialstable}
 Let $f\in\CC[x_1,\ldots,x_n]$ be stable and $F$ a face of $\newt(f)$. Then $f_F$ is stable as well.
\end{prop}

\begin{proof}
 Let $f=\sum_{\alpha\in\mathcal{S}}c_\alpha x^{\alpha}$ with support $\mathcal{S}\subseteq \ZZ_{\geq 0}^{n}$  for $c_\alpha\in\CC$.
 Let $a\in\RR^n$ and $c\in\RR$ such that $\langle a,x\rangle\geq c$ for all $x\in\newt(f)$ with equality exactly when $x\in F$. For all $\epsilon>0$ the polynomial $$f_\epsilon:=\epsilon^{-c}\cdot f(\epsilon^{a_1}x_1,\ldots,\epsilon^{a_n}x_n)=\sum_{\alpha\in\ZZ^{n}\cap\newt(f)}\epsilon^{\langle a, \alpha \rangle-c}c_\alpha x^{\alpha}$$is also stable. It follows that $f_F=\lim_{\epsilon\to0}f_\epsilon$ is stable by Hurwitz's theorem.
\end{proof}

For any $f\in\CC[x_1,\ldots,x_n]$ we denote by $f^{\#}$ and $f_{\#}$ its \emph{leading form} resp.~its \emph{initial form}, i.e., the sum of all terms of largest or lowest degree respectively. If $f$ is stable, then both $f^{\#}$ and $f_{\#}$ are stable. This has been proved in \cite[\S 2.2]{Choe_2004} but it also follows from applying \Cref{prop:facialstable} to the face of $\newt(f)$ that maximizes resp.~minimizes the all-ones vector.

The standard way of calculating the basis generating polynomial of a minor of the matroid $M$ is to delete and contract one element at a time. If $i\in E$ is not a coloop of $M$, then $h_{M\setminus\{e\}}=h_{M}|_{x_i=0}$ and $h_{M\setminus\{e\}}=\frac{\partial h_{M}}{\partial x_i}$ otherwise. If $i\in E$ is not a loop of $M$, then $h_{M/\{e\}}=\frac{\partial h_{M}}{\partial x_i}$ and $h_{M/\{e\}}=h_{M}$ otherwise. Since these operations preserve the half-plane property, this property is minor-closed \cite[\S 4.1]{Choe_2004}. The basis generating polynomials of minors can also be expressed in terms of certain initial and leading forms.

\begin{lem}\label{lem:polysofminors}
 For any $S\subseteq E$ we have:
 \begin{enumerate}
  \item $h_{M\slash S}=c\cdot(h_M|_{x_i=1\textrm{ for }i\in S})_{\#}$ for some constant $c$.
  \item $h_{M\setminus S}=c'\cdot(h_M|_{x_i=1\textrm{ for }i\in S})^{\#}$ for some constant $c'$.
 \end{enumerate}
\end{lem}

\begin{proof}
 We first prove $(1)$. By \cite[Cor.~3.1.8]{oxley} a subset $I\subseteq E\setminus S$ is a basis of $M/S$ if and only if there is a basis $B$ of $M$ such that $B\setminus I$ is a  basis of $M|_S$, i.e., a maximal independent subset of $S$. This shows that the support of $h_{M/S}$ agrees with the support of $f:=(h_M|_{x_i=1\textrm{ for }i\in S})_{\#}.$ It also follows from \cite[Cor.~3.1.8]{oxley} that the coefficient of each monomial in $f$ is the number of bases of $M|_S$.
 
 By the same argument we see that the support of $h_{M\setminus S}$ agrees with the support of $g:=(h_M|_{x_i=1\textrm{ for }i\in S})^{\#}$ and that the coefficient of every monomial in $g$ is the number of bases of $M/(E\setminus S)$. This shows $(2)$.
\end{proof}

We will now study some concepts related to the half-plane property that have been of recent interest.

\begin{Def}
  A real homogeneous polynomial $h\in\RR[x_1,\ldots,x_n]$ of degree $d$ is said to have a \emph{determinantal representation} if there are positive semi-definite matrices $A_{1},\ldots,A_{n}$ of size $d\times d$ such that 
  \[
    h=\det(x_{1}A_{1}+\cdots+x_{n}A_{n}).
  \]
  We say that $f$ is \emph{weakly determinantal} if $f^r$ has a determinantal representation for some suitable $r\in\NN$. A matroid is called \emph{weakly determinantal} if its basis generating polynomial is weakly determinantal.
\end{Def}

Every polynomial that has a determinantal representation is stable. The same is true for weakly determinantal polynomials. Furthermore, the basis generating polynomial of a matroid $M$ has a determinantal representation if and only if $M$ is a regular matroid \cite[\S 8]{Choe_2004}. The class of weakly determinantal matroids is less well understood. We know that it includes all $\sqrt[6]{1}$-matroids by \cite[\S 8]{Choe_2004}. The main result of \cite{Branden2011} says that the V\'amos matroid is not weakly determinantal although it has the half-plane property by \cite{Wagner_2009}. Further such matroids were found in \cite{Amini2018} and in \Cref{sec:classi} we will also give some new examples. Of course, being regular is a minor-closed property. More generally, we have:

\begin{thm}[Cor.~5.8 in \cite{Kummer-2013-hyperbolic-polynomials-interlacers}]\label{prm2}
 Let $f\in\RR[x_1,\ldots,x_n]$ be a multiaffine stable polynomial. If $f$ has a determinantal representation, then $\frac{\partial f}{\partial x_k}$ and $f|_{x_k=0}$ also have determinantal representations.
\end{thm}

We will prove the analogous statement for weakly determinantal polynomials in \Cref{cor:weaklyclosed}.
The following characterization of multiaffine stable polynomials is due to Br\"and\'en.

\begin{thm}[Thm.~5.6 in \cite{Branden2007}]\label{thm:rayleighbrand}
  Let $h\in\RR[x_1,\ldots,x_n]$ be multiaffine. The following are equivalent:
  \begin{enumerate}
    \item  $h$ is stable,
    \item for all $1\leq i,j\leq n$, the \emph{Rayleigh difference}
    \[
      \Delta_{ij}(h):= \frac{\d{h}}{\d{x_{i}}}\cdot\frac{\d{h}}{\d{x_{j}}}
      -\frac{\d^{2}{h}}{\d{x_{i}}\d{x_{j}}}\cdot h
    \]
    is nonnegative on $\RR^{n}$.
  \end{enumerate}
\end{thm}


\begin{Def}
 A multiaffine polynomial $h\in\RR[x_1,\ldots,x_n]$ is called \emph{SOS-Rayleigh} if for all $1\leq i,j\leq n$, the {Rayleigh difference}
    $\Delta_{ij}(h)$ is a sum of squares (of polynomials). A matroid is called \emph{SOS-Rayleigh} if its basis generating polynomial is SOS-Rayleigh.
\end{Def}

For us the notion of being SOS-Rayleigh is important for several reasons. It follows directly from the definition and \Cref{thm:rayleighbrand} that SOS-Rayleigh implies the half-plane property. However, it is easier to test whether a polynomial is a sum of squares than checking for global nonnegativity. Namely, the former can be done efficiently by running a semidefinite program \cite{2012-siamSDP}. Moreover, it was shown in \cite{Kummer-2013-hyperbolic-polynomials-interlacers} that being weakly determinantal implies SOS-Rayleigh. Via testing for the SOS-Rayleigh property we can thus find new examples of matroids with the half-plane property that are not weakly determinantal. Finally, polynomials with the half-plane property that are not SOS-Rayleigh are of interest in the context of Saunderson's approach to certifying the global nonnegativity of multivariate polynomials using \emph{hyperbolic programming} \cite{soshyperbolic}. We first show that in the multiaffine case being SOS-Rayleigh is preserved under taking derivatives.

\begin{thm}
 Let $f\in\RR[x_1,\ldots,x_n]$ be a multiaffine stable polynomial. If $f$ is SOS-Rayleigh, then $\frac{\partial f}{\partial x_k}$ and $f|_{x_k=0}$ are SOS-Rayleigh as well.
\end{thm}

\begin{proof}
 Let $g=\frac{\partial f}{\partial x_k}$ and $h=f|_{x_k=0}$. One calculates $$\Delta_{ij}f=x_k^2\cdot\Delta_{ij}g+x_k\cdot p+\Delta_{ij}h$$for some $p\in\RR[x_1,\ldots,x_n]$ that does not depend on $x_k$. By assumption $\Delta_{ij}f=\sum_{l=1}^ms_l^2$ is a sum of squares. Comparing degrees on both sides we get that $s_l=a_lx_k+b_l$ for some polynomials $a_l,b_l$ that do not depend on $x_k$. It follows that $\Delta_{ij}g=\sum_{l=1}^ma_l^2$ and $\Delta_{ij}h=\sum_{l=1}^mb_l^2$.
\end{proof}

\begin{cor}
 The class of SOS-Rayleigh matroids is minor-closed.
\end{cor}
Let $h\in\RR[x_{1},\dots,x_{n}]$ be a homogeneous multiaffine polynomial. We can summarize the relations between the properties we defined so far as follows:
\begin{center}
\begin{tikzpicture}
  \node(a) at (0,0){$\begin{array}{c} h \text{ is weakly} \\ \text{determinantal}\\ \end{array}$};
  \node(b) at (1.8,0){$\implies$};
  \node(c) at (3.5,0){$h$ has HPP};
  \node(d) at (5,0){$\iff$};
  \node(e) at (7.5,0){$\begin{array}{c} h \text{ is hyperbolic} \\ \text{with resp. to all }e\in \RR^{n}_{\geq 0}\\ \end{array}$};

  \node[rotate=270](f) at (3.5,-1){$\iff$};
  \node[right](i) at (3.5,-1){\cite{Branden2007}};
  \node(g) at (3.5,-2){$\begin{array}{c} \Delta_{ij} h \geq 0 \\ \text{for all }i,j\in\left[ n \right]\\ \end{array}$};
  \node[rotate=270](i) at (0,-1){$\implies$};
  \node[right](i) at (0,-1){\cite{Kummer-2013-hyperbolic-polynomials-interlacers}};
  \node(j) at (-0.5,-2){$h$ is SOS-Rayleigh};
  \node[rotate=270](k) at (3.5,1){$\implies$};
  \node(l) at (3.5,2){$\begin{array}{c} h \text{ has a determinantal} \\ \text{representation}\\ \end{array}$};
  \node[rotate=230](m) at (1.2,1.2){$\implies$};
  \node(n) at (6,2){$\iff$};
   \node[above](n) at (6,2){\cite{Kummer-2013-hyperbolic-polynomials-interlacers}};
  \node(p) at (8,2){$\begin{array}{c} \Delta_{ij} h \text{ is a square }\\ \text{for all }i,j\in\left[ n \right]\\ \end{array}$};
  \node(r) at (1.5,-2){$\implies$};
  \end{tikzpicture}
\end{center}

\section{Hyperbolic polynomials and spectrahedra}\label{sec:hypspec}

\begin{Def}
  A polynomial $h$ is called \emph{hyperbolic} with respect to $e\in \RR^{n}$ if $h(e)\neq 0$ and  for all $v\in \RR^{n}$, the univariate polynomial $h(v-te)\in \RR[t]$ has only real roots.
 The hyperbolicity cone of $h$ at $e$ is
  \[
    C_{h}(e)=\left\{ v\in \RR^{n}: h(v-te)=0 \Rightarrow t\in \RR_{\geq 0} \right\}.
  \]
\end{Def}

Every homogeneous stable polynomial $f\in\RR[x_1,\ldots,x_n]$ is hyperbolic with respect to the all-ones vector and one can study the associated hyperbolicity cone. One can show that it always contains the positive orthant. Conversely, if  the hyperbolicity cone of a hyperbolic polynomial contains the postive orthant, then it is stable.
Moreover, if the non-zero coefficients of a homogeneous polynomial have no sign variation, then being hyperbolic with respect to the all-ones vector implies that the associated hyperbolicity cone contains the positive orthant.  
  In particular, the basis generating polynomials of a matroid is stable if and only if it is hyperbolic with respect to the all-ones vector.

\begin{Def}
  A \emph{spectrahedral cone} is a cone of the form
  \[
    C=\left\{v\in \RR^{n}: A(v)=v_{1}A_{1}+\ldots+ v_{n}A_{n}\succeq 0\right\}
  \]
  for some real symmetric matrices $A_{1},\ldots, A_{n}$ of size $d\times d$, where $A(v)\succeq 0$ means that $A(v)$ is positive semi-definite. We say that a matroid is \emph{spectrahedral} if it has the half-plane property and the hyperbolicity cone at the all-ones vector of its basis generating polynomial is a spectrahedral cone.
\end{Def}

Recall the following criterion for a hyperbolicity cone to be spectrahedral.

\begin{thm}[Thm.~2.1 in \cite{Vinnikov_2012}]\label{mariospec}
 Let $h\in \RR\left[x_{1},\dots,x_{n}\right]$ be a homogeneous stable polynomial and $e=(1,\ldots,1)$ the all-ones vector. The hyperbolicity cone $C_{h}(e)$ is spectrahedral if and only if there is another homogeneous stable polynomial $g\in \RR \left[x_{1},\dots,x_{n}\right]$, such that $C_{h}(e)\subseteq C_{g}(e)$ and $h\cdot g$ has a determinantal representation.
\end{thm}

\begin{lem}\label{rootsfg}
  Let $f$, $g$ be two polynomials that are hyperbolic with respect to $e\in\RR^{n}$. For $v\in \RR^{n}$ let $g_{\mathrm{min}}(v)$ be the smallest root of $g(te-v)$ and $f_{\mathrm{min}}(v)$ be the smallest root of $f(te-v)$ respectively. Then we have
  $f_{\mathrm{min}}(v)\leq g_{\mathrm{min}}(v)$ for all $v\in\RR^n$ if and only if $C_{f}(e)\subseteq C_{g}(e)$.
\end{lem}
  
\begin{proof}
  $\Leftarrow:$ Assume that $C_{f}(e)\subseteq C_{g}(e)$.
  Then for all $v\in\RR^{n}$ the smallest root of $$f((t+f_{\textrm{min}}(v)) e-v)$$ is zero. Thus $f(te-(-f_{\textrm{min}}(v)e+v))$ has only non-negative zeros. By definition of $C_f(e)$ this implies that $-f_{\textrm{min}}(v)e+v\in C_{f}(e)\subseteq C_{g}(e)$. 
  By the definition of $C_{g}(e)$ this shows that $g(te-(-f_{\textrm{min}}(v)e+v))=g((t+f_{\textrm{min}}(v))e-v)$ has only non-negative zeros. Thus $f_{\textrm{min}}(v)\leq g_{\textrm{min}}(v)$.
  
  $\Rightarrow:$  For $v\in C_{f}(e)$ we have by definition that $f_{\textrm{min}}(v)\geq 0$. If $f_{\textrm{min}}(v)\leq g_{\textrm{min}}(v)$, then $g_{\textrm{min}}(v)\geq0$ and thus $v\in C_{g}$.
\end{proof}

\begin{lem}\label{limitcont}
 Let $(h_{i})_{i\in\NN}, (g_{i})_{i\in\NN}$ two sequences of polynomials that are hyperbolic with respect to $e\in\RR^{n}$. Assume that we have $h=\lim_{i\rightarrow \infty}h_{i}$ and $g=\lim_{i\rightarrow \infty}g_{i}$ for polynomials $g,h$ that are hyperbolic with respect to $e$. If $C_{h_{i}}(e)\subseteq C_{g_{i}}(e)$ for all $i\in\NN$, then $C_{h}(e)\subseteq C_{g}(e)$.
\end{lem}
  
\begin{proof}
 For every $v\in\RR^n$ and $i\in\NN$ the smallest root of $h_i(te-v)$ is at most the smallest root of $g_i(te-v)$ by \Cref{rootsfg}. Since the roots of a univariate polynomial depend continuously on the coefficients, we have the same for $h$ and $g$. Thus, again by \Cref{rootsfg} $C_{h}(e)\subseteq C_{g}(e)$.
\end{proof} 

\begin{lem}\label{lem:lemyk}
 Let $h\in \RR\left[x_{1},\dots,x_{n}\right]$ be a homogeneous polynomial that does not depend on $x_{k}$. Let $h':=x_{k}^mh$ for a positive integer $m$, and assume that $h'$ (and thus $h$) is hyperbolic with respect to $e\in\RR^n$.
 Then $$C_{h}(e)=\{v\in\RR^n:\,  v+\lambda e_k\in C_{h'}(e) \text{ for some }\lambda\in\RR\}.$$
\end{lem}

\begin{proof}
 We have $C_{h'}(e)=C_h(e)\cap C_{x_k}(e)\subseteq C_{h}(e)$. For every $v\in C_{h}(e)$ we have $v+\lambda e_k\in C_{h}(e)$ since $h$ does not depend on $x_k$. This shows the inclusion $\supset$.

 On the other hand, if $v\in C_{h}(e)$, then for $\lambda=-v_k$, the polynomial $h'(te-(v+\lambda e_k))$ equals $c\cdot t^m\cdot h(te-v)$ for some non-zero constant $c\in\RR$ and thus has only nonnegative roots. This shows the other inclusion. 
\end{proof}

\begin{cor}\label{lemxk}
  Let $h_1,h_2\in \RR\left[x_{1},\dots,x_{n}\right]$ be homogeneous polynomials that do not depend on $x_{k}$. Let $h_i':=x_{k}^{m_i}h$ and assume that $h_i'$ (and thus $h_i$) is hyperbolic with respect to $e\in\RR^n$ for $i=1,2$. If $C_{h_1'}(e)\subseteq C_{h_2'}(e)$, then $C_{h_1}(e)\subseteq C_{h_2}(e)$.
\end{cor}
  
\begin{proof}
 This follows from the description of $C_{h_i}(e)$ given in \Cref{lem:lemyk}.
\end{proof}

\begin{cor}\label{cor:sharpcones}
 Let $h_1,h_2\in \RR\left[x_{1},\dots,x_{n}\right]$ be homogeneous stable polynomials and $e=(1,\ldots,1)$ the all-ones vector. Let $C_{h_1}(e)\subseteq C_{h_2}(e)$ and $h_i'=(h_i|_{x_k=1})_\#$ for $i=1,2$. Then $C_{h_1'}(e)\subseteq C_{h_2'}(e)$.
\end{cor}

\begin{proof}
  Let $d_i$ be the degree of $h_i$ and $r_i$ be the smallest degree of a monomial of $h_i|_{x_{k}=1}$. For $\gamma>0$ let $T_\gamma$ the linear map defined by
  \[
   T_\gamma(x_1,\ldots,x_n)=(\gamma x_{1},\ldots,\gamma x_{k-1},x_{k},\gamma x_{k+1},\ldots, \gamma x_{n}).
  \]
For $i=1,2$ the polynomial $h_{i,\gamma}=\gamma^{-r_i}h_i(T_\gamma x)$ is stable. Since $T_\gamma^{-1}e$ is in the positive orthant and thus in the interior of $C_{h_i,\gamma}(e)$, we have 
\[
 C_{h_{i,\gamma}}(e)=C_{h_{i,\gamma}}(T^{-1}_\gamma e)=\{v\in\RR^n: h_i(T_\gamma v-te)=0\Rightarrow t\geq0\}=T_\gamma^{-1}(C_{h_i}(e)),
\]
where the first equality is \cite[Theorem~2]{gar}.
In particular $C_{h_1'}(e)\subseteq C_{h_2'}(e)$ implies $T_\gamma^{-1}(C_{h_1'}(e))\subseteq T_\gamma^{-1}(C_{h_2'}(e))$ and therefore 
   $C_{h_{1,\gamma}}(e)\subseteq C_{h_{2,\gamma}}(e)$. Since $h_i$ is homogeneous, the  limit $\lim_{\gamma\rightarrow 0}h_{i,\gamma}$ kills all the monomials, except those that are divisible by $x_{k}^{d_i-r_i}$. Thus, we have    \[
      \lim_{\gamma\rightarrow 0}h_{i,\gamma}   = x_{k}^{d_i-r_i}\left(h_i|_{x_{k}=1}\right)_{\#}.
    \]
 Now the claim follows from \Cref{limitcont} and \Cref{lemxk}.
\end{proof}

\begin{thm}\label{derspec}
  Let $h\in \RR\left[x_{1},\dots,x_{n}\right]$ be a homogeneous, multiaffine and stable polynomial and $e=(1,\ldots,1)$ the all-ones vector. If $\frac{\partial h}{\partial x_{k}}\neq0$ and the hyperbolicity cone $C_{h}(e)$ is spectrahedral, then the hyperbolicity cone of $\frac{\partial h}{\partial x_{k}}$ is also spectrahedral.
\end{thm}  

\begin{proof}
    Let $h\in \RR \left[ x_{1},\dots,x_{n}\right]$ be hyperbolic with respect to $e\in\RR^{n}$ and assume that $C_{h}(e)$ is spectrahedral. Then, by \Cref{mariospec}, there exists a stable polynomial $g\in \RR \left[ x_{1},\dots,x_{n}\right]$ such that $C_{h}(e)\subseteq C_{g}(e)$ and $p:=h\cdot g$ has a determinantal representation. We can assume that $k=1$ and that
    \[
      0\neq p=\det(x_{1}A_{1}+\ldots+x_{n}A_{n})
    \]
    where $A_{i}$ is positive semidefinite of rank $r_{i}\geq1$. Please note that $p$ does not need to be multiaffine. Let us write each $A_{i}$ as the sum of rank $1$ matrices:
    \[
      A_{i}=A_{i1}+\ldots+ A_{ir_{i}},
    \]
    where $A_{ij}$ are positive semidefinite matrices of rank $1$,
    $i\in \left[ n \right],\, j\in\left[r_{i}\right]$. Let us define a new polynomial
    \[
      \tilde{p}=
      \det(x_{11}A_{11}+x_{12}A_{12}+\dots+x_{1r_{1}}A_{1r_{1}}+\dots+x_{nr_{n}}A_{nr_{n}})
    \]
    in $\sum_{i=1}^{n}r_{i}$ variables. The degree of each variable $x_{i{j}}$ in $\tilde{p}$ is the rank of $A_{i{j}}$ so that $\tilde{p}$ is multiaffine. Moreover, since the rank of $A_i$ is $r_i$, there are monomials of $\tilde{p}$ that are divisible by $x_{i1}\cdots x_{ir_i}$, see e.g. \cite[p.~1206]{Branden2011}. By \Cref{prm2} the polynomial $$\frac{\partial^{r_1}\tilde{p}}{\partial x_{11}\cdots\partial x_{1r_1}}=(\tilde{p}|_{x_{1j}=1 \textrm{ for }j\in[r_1]})_\#$$ has a determinantal representation. The same is thus true for the polynomial $\hat{p}$ that we obtain from $\frac{\partial^{r_1}\tilde{p}}{\partial x_{11}\cdots\partial x_{1r_1}}$ by setting $x_{ir_{1}}=\cdots=x_{ir_{k}}=x_{i}$ for all $i\in \left[ n \right]$. By construction we have that $$\hat{p}=(p|_{x_1=1})_\#=(h|_{x_1=1}\cdot g|_{x_1=1})_\#=(h|_{x_1=1})_\#\cdot (g|_{x_1=1})_\#.$$Since $h$ is multiaffine, we have $(h|_{x_1=1})_\#=\frac{\partial h}{\partial x_{1}}$. Thus in order to prove the claim it remains to show that the hyperbolicity cone of $(h|_{x_1=1})_\#$ is contained in the hyperbolicity cone of $(g|_{x_1=1})_\#$. This is the content of \Cref{cor:sharpcones}.
\end{proof}

\begin{rem}
 The proof of \Cref{derspec} crucially relies on the assumption that $h$ is multiaffine and it is not known whether one can drop this assumption. A related result was proved in \cite{kummer2020spectral}: If $h$ (not necessarily multiaffine) has a determinantal representation, then the hyperbolicity cone of any iterated derivative in direction $e$ of $h$ is spectrahedral.
\end{rem}

\begin{lem}\label{lem:deletspec}
 Let $h\in \RR\left[x_{1},\dots,x_{n}\right]$ be a homogeneous, multiaffine and stable polynomial, and let $e=(1,\ldots,1)$ the all-ones vector. If $h|_{x_{k}=0}\neq0$ and the hyperbolicity cone $C_{h}(e)$ is spectrahedral, then the hyperbolicity cone of $h|_{x_{k}=0}$ is also spectrahedral.
\end{lem}

\begin{proof}
 The hyperbolicity cone of $h|_{x_{k}=0}$ is the intersection of $C_{h}(e)$ with the hyperplane $x_k=0$ and thus spectrahedral.
\end{proof}

\begin{cor}
 The class of spectrahedral matroids is minor-closed.
\end{cor}

\begin{proof}
 If a matroid is spectrahedral, then the same is true for its contractions by \Cref{derspec} and for its deletions by \Cref{lem:deletspec}. 
\end{proof}

\begin{cor}\label{cor:weaklyclosed}
 Let $h\in \RR\left[x_{1},\dots,x_{n}\right]$ be  multiaffine, stable and weakly determinantal. Then $\frac{\partial h}{\partial x_{k}}$ and $h|_{x_k=0}$ are also weakly determinantal.
\end{cor}

\begin{proof}
 For $\frac{\partial h}{\partial x_{k}}$ we can apply the proof of \Cref{derspec} to the case $g=h^r$ for some $r\geq0$. The claim for $h|_{x_k=0}$ follows directly from the definition.
\end{proof}

\begin{cor}
 The class of weakly determinantal matroids is minor-closed.
\end{cor}

\section{Matroid Polytopes} \label{sec1}

\begin{Def}
  The \emph{matroid polytope} $P(M)$ of a matroid $M$ is defined as
  \begin{align*}
    P(M)&=\conv\left\{\sum_{i\in B}e_i: B\in \mathcal{B}\right\}
  \end{align*}  where $e_i\in\RR^n$ denotes the $i$th unit vector.
\end{Def}

\begin{rem}
 The matroid polytope of a matroid is the Newton polytope of its basis generating polynomial.
\end{rem}

\begin{rem}\label{tp}
 It follows for example from \cite[\S4.4]{murota} that there is the following description of $P(M)$ by means of inequalities:
 \begin{align*}
    P(M)      &=\left\{x\in r\cdot\Delta_{n}: \sum_{i\in S} x_i \leq \rank(S) \text{ for all flats } S\subseteq E \right\}
  \end{align*}  where $r\cdot\Delta_{n}\subseteq\RR^n$ denotes $r$-fold dilation of the standard $n$-simplex $\Delta_n\subseteq\RR^n$. It follows from \cite[Thm.~4.1]{Gelfand_1987} that any face $F$ of $P(M)$ is the matroid polytope of a suitable uniquely determined matroid $M_F$ itself. For a face $F\neq\emptyset$ of $P(M)$ of codimension $1$ there are two possibilities. Either there is a flat $S$ of $M$ such that $F$ consists of all points $x\in r\cdot\Delta_{n}$ that satisfy $\sum_{i\in S} x_i = \rank(S)$. In this case we have $M_F=M|_S\oplus M\slash S=M\setminus \left(E\setminus S\right) \oplus M\slash S$ by \cite[Thm.~2]{Gelfand2}.
  Otherwise $F$ is the intersection of $r\cdot\Delta_{n}$ with a coordinate hyperplane $x_i=0$. Letting $S=E\setminus\{i\}$ we have $\rank(S)=r$ since $F\neq\emptyset$ and $F$ consists of all points $x\in r\cdot\Delta_{n}$ that satisfy $\sum_{i\in S} x_i = \rank(S)$. Again we have $M_F=M|_S=M|_S\oplus M/S$ although $S$ is not a flat.
\end{rem}

\begin{cor}\label{TT}
 Let $M$ be a matroid and $F$ a face of its matroid polytope. Let $P$ be a property that is minor-closed and preserved under taking the direct sum of matroids. If $M$ has property $P$, then $M_F$ has property $P$.
 
 This applies to the half-plane property, being weakly determinantal, spectrahedral and SOS-Rayleigh.
\end{cor}

\begin{proof}
 It suffices to show the claim for facets $F$ of the matroid polytope. Then it follows from the description $M_F=M|_S\oplus M\slash S$ from \Cref{tp}.
 
 The second claim is true for the half-plane property by \Cref{prop:facialstable}. The other three properties are minor-closed and it remains to show that these are preserved under taking the direct sums. For being weakly determinantal and spectrahedral this follows from taking suitable block matrices. For being SOS-Rayleigh it follows from the identity $\Delta_{ij}(f\cdot g)=f^2\Delta_{ij}g+g^2\Delta_{ij}f$.
\end{proof}

\begin{cor}\label{T1}
  Let $M$ be a matroid and $F$ be a facet of $P(M)$. There is a subset $S\subseteq E$ and a constant $c$ such that 
  \[
  h_{M_F}=c\cdot(h_M|_{x_i=1\textrm{ for }i\in E\setminus S})_{\#} \cdot (h_M|_{x_i=1\textrm{ for }i\in S})^{\#}.
  \]
\end{cor}

\begin{proof}
 By \Cref{tp} we have $M_F=M|_S\oplus M/S$ and thus the claim follows from \Cref{lem:polysofminors}.
\end{proof}

\begin{ex}
  Let $M$ be the matroid with basis generating polynomial
  \[
    h_{M}=x_{1}x_{2}+x_{2}x_{3}+x_{1}x_{4}+x_{2}x_{4}+x_{3}x_{4}.
  \]
  Its matroid polytope $P(M)$ is depicted in \Cref{pm}.
  \begin{figure}
    \begin{center}
      \begin{subfigure}[b]{0.4\linewidth}  
        \begin{tikzpicture}[x=0.6cm, y=0.6cm,every node/.style={node font=\footnotesize}]
          \node[below, color=gray] (e4) at (0,1) {};
          \node [below, color=gray ] (e2) at (6,0.5) {};
          \node [above, color=gray] (e3) at (1.8,5) {};
          \node [right,  color=gray](e1) at (4.4,5.7)  {};
          \node [below, color=gray] (4a) at (e4) {$\left(0,0,0,2\right)$};
          \node [below, color=gray] (2a) at (e2) {$\left(0,2,0,0\right)$};
          \node [above, color=gray] (3a) at (e3) {$\left(0,0,2,0\right)$};
          \node [above, color=gray] (1a) at (e1) {$\left(2,0,0,0\right)$};
          
          \coordinate[below left, color=gray] (ce4) at (0,1) {};
          \coordinate [below right ] (ce2) at (6,0.5) {};
          \coordinate[above left] (ce3) at (1.8,5) {};
          \coordinate [below right](ce1) at (4.4,5.7)  {};
          
          \draw [ color=gray] (ce4) -- (ce3) -- (ce2) -- cycle;
          \draw [ color=gray]  (ce2) -- (ce1);
          \draw [ color=gray]  (ce3) -- (ce1);
          \draw [ color=gray] (ce4) -- (ce1);
          
          \coordinate[] (c24) at ($(ce2)!0.5!(ce4)$){};   
          \coordinate [ ] (c23) at ($(ce2)!0.5!(ce3)$) {};  
          \coordinate [ ] (c14) at ($(ce1)!0.5!(ce4)$) {}; 
          \coordinate [] (c12) at ($(ce1)!0.5!(ce2)$) {};
          \coordinate [left] (c34) at ($(ce3)!0.5!(ce4)$) {};
          
          \draw [fill=cyan, opacity=0.5] (c34) -- (c14) -- (c12) -- (c23) -- (c34);
          \draw [fill=cyan, opacity=0.5] (c34) -- (c14) -- (c24) -- (c34);
          \draw [fill=cyan, opacity=0.5, dotted] (c14) -- (c12) -- (c24) -- (c14);
          \draw [fill=cyan, opacity=0.5] (c12) -- (c23) -- (c24) -- (c12);
          \draw [fill=cyan, opacity=0.5] (c34) -- (c23) -- (c24) -- (c34);
          
          \node [below] (24) at ($(ce2)!0.5!(ce4)$) {$\mathbf{v_{42}}$};
          \node [below, color=violet] (24a) at (24) {$\left(0,1,0,1\right)$};
          \node [below] (23) at ($(ce2)!0.5!(ce3)$) {$\mathbf{v_{23}} $};
          \node [below, color=violet] (23a) at (23) {$\left(0,1,1,0\right)$};
          \node [above] (14) at ($(ce1)!0.5!(ce4)$) {$\mathbf{v_{14}}$};
          \node [below, color=violet] (14a) at (14) {$\left(1,0,0,1\right)$};
          \node [right] (12) at ($(ce1)!0.5!(ce2)$) {$\mathbf{v_{12}}$};
          \node [below, color=violet] (12a) at (12) {$\left(1,1,0,0\right)$};
          \node [left] (34) at ($(ce3)!0.5!(ce4)$) {$\mathbf{v_{34}}$};
          \node [below, color=violet] (34a) at (34) {$\left(0,0,1,1\right)$};
        \end{tikzpicture}
        \caption{$P(M)$}
        \label{pm}
      \end{subfigure}
      \begin{subfigure}[b]{0.4\linewidth}
        \begin{tikzpicture}[x=0.6cm, y=0.6cm,every node/.style={node font=\footnotesize}]
          \node[below, color=gray] (e4) at (0,1) {};
          \node [below, color=gray ] (e2) at (6,0.5) {};
          \node [above, color=gray] (e3) at (1.8,5) {};
          \node [right,  color=gray](e1) at (4.4,5.7)  {};
          \node [below, color=gray] (4a) at (e4) {$\left(0,0,0,2\right)$};
          \node [below, color=gray] (2a) at (e2) {$\left(0,2,0,0\right)$};
          \node [above, color=gray] (3a) at (e3) {$\left(0,0,2,0\right)$};
          \node [above, color=gray] (1a) at (e1) {$\left(2,0,0,0\right)$};
          
          \coordinate[below left, color=gray] (ce4) at (0,1) {};
          \coordinate [below right ] (ce2) at (6,0.5) {};
          \coordinate[above left] (ce3) at (1.8,5) {};
          \coordinate [below right](ce1) at (4.4,5.7)  {};
          
          \draw [ color=gray] (ce4) -- (ce3) -- (ce2) -- cycle;
          \draw [ color=gray]  (ce2) -- (ce1);
          \draw [ color=gray]  (ce3) -- (ce1);
          \draw [ color=gray] (ce4) -- (ce1);
          
          \coordinate[] (c24) at ($(ce2)!0.5!(ce4)$){};   
          \coordinate [ ] (c23) at ($(ce2)!0.5!(ce3)$) {};  
          \coordinate [ ] (c14) at ($(ce1)!0.5!(ce4)$) {}; 
          \coordinate [] (c12) at ($(ce1)!0.5!(ce2)$) {};
          \coordinate [left] (c34) at ($(ce3)!0.5!(ce4)$) {};
          
          \draw [fill=magenta, opacity=0.8] (c34) -- (c14) -- (c12) -- (c23) -- (c34);
          \draw [fill=cyan, opacity=0.2] (c34) -- (c14) -- (c24) -- (c34);
          \draw [fill=cyan, opacity=0.2, dotted] (c14) -- (c12) -- (c24) -- (c14);
          \draw [fill=cyan, opacity=0.2] (c12) -- (c23) -- (c24) -- (c12);
          \draw [fill=cyan, opacity=0.2] (c34) -- (c23) -- (c24) -- (c34);
          
          \node [below, color=gray] (24) at ($(ce2)!0.5!(ce4)$) {$\mathbf{v_{42}}$};
          \node [below, color=gray] (24a) at (24) {$\left(0,1,0,1\right)$};
          \node [below] (23) at ($(ce2)!0.5!(ce3)$) {$\mathbf{v_{23}} $};
          \node [below, color=violet] (23a) at (23) {$\left(0,1,1,0\right)$};
          \node [above] (14) at ($(ce1)!0.5!(ce4)$) {$\mathbf{v_{14}}$};
          \node [below, color=violet] (14a) at (14) {$\left(1,0,0,1\right)$};
          \node [right] (12) at ($(ce1)!0.5!(ce2)$) {$\mathbf{v_{12}}$};
          \node [below, color=violet] (12a) at (12) {$\left(1,1,0,0\right)$};
          \node [left] (34) at ($(ce3)!0.5!(ce4)$) {$\mathbf{v_{34}}$};
          \node [below, color=violet] (34a) at (34) {$\left(0,0,1,1\right)$};
        \end{tikzpicture}
        \caption{$P(M_{F})$}
        \label{pmf}
      \end{subfigure}
    \end{center}
    \caption{Matroid polytopes $P(M)$ and $P(M_{F})$}
    \label{matroidpolytopes}
  \end{figure}
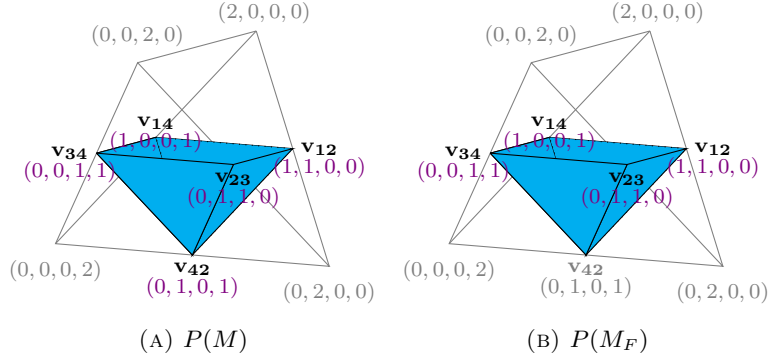
  The nonempty proper flats of $M$ are $S_{1}=\left\{ 2 \right\}$, $S_{2}=\left\{ 4 \right\}$ and $S_{3}=\left\{ 1,3 \right\}$. Thus $P(M)$ has the following representation in terms of inequalities:
  \[
    P(M)=\left\{ x\in 2\Delta_{4}: \left( \begin{array}{cccc} 0&1&0&0\\0&0&0&1 \\1&0&1&0\\1&1&1&1\end{array} \right) \begin{pmatrix}
                    x_1\\ x_2 \\ x_3 \\ x_4                                                                                                                  \end{pmatrix}
\leq \left( \begin{array}{c} 1\\1\\1\\2\end{array}\right) \right\}. 
  \]
  Consider for instance the flat $\left\{ 1,3 \right\}$, the hyperplane $H=\left\{x\in \RR^{4}:x_{1}+x_{3}=1  \right\}$ and the face $F=P(M)\cap H$. The matroid $M_{F}$ is the matroid whose bases are encoded by the vertices of the face $F$ that is illustrated in Figure~\ref{pmf}.
  Its basis generating polynomial is
  \[
    h_{M_{F}}=x_{1}x_{2}+x_{2}x_{3}+x_{1}x_{4}+x_{3}x_{4}=\left( x_{1}+x_{3} \right)\left(x_{2}+x_{4}\right).
  \]
  which is consistent with \Cref{T1}.
\end{ex}  

For the rest of the section we will extend these results to arbitrary stable polynomials. In \Cref{prop:facialstable} we have already seen that given a stable polynomial $f$ and a face $F$ of its Newton polytope, the polynomial $f_F$ is also stable.

\begin{lem}\label{lem:gotofacemuaff}
 Let $f\in\RR[x_1,\ldots,x_n]$ be homogeneous and multiaffine. Let $F$ be a face of $\newt(f)$. If $f$ has a determinantal representation, then $f_F$ has a determinantal representation as well.
\end{lem}

\begin{proof}
 Without loss of generality, we can assume that $F$ is a facet of $P:=\newt(f)$. Since $f$ is multiaffine, we can write $$f=\det(x_1A_1+\ldots+x_nA_n)$$ for some positive semidefinite matrices of rank at most $1$, i.e., we have $A_i=a_i\cdot a_i^t$ for some $a_i\in\RR^d$ \cite[\S 3]{Branden2011}. Let $M$ be the matroid represented by the vectors $a_1,\ldots,a_n$. Then we have $P=P(M)$. 
 By \Cref{tp} there is a subset $S$ of the ground set of $M$ such that $F$ consists of all $p\in d\cdot\Delta_{n}$ that satisfy $\sum_{i\in S} p_i = \rank(S)=:r$. After replacing the $A_i$ by $N^tA_iN$ for some invertible matrix $N$, we can assume that the span of all $a_i$ with $i\in S$ is spanned by the first $r$ standard unit vectors. This means that in the matrix $A:=x_1A_1+\ldots+x_nA_n$ the variables $x_i$ for $i\in S$ only show up in the upper left $r\times r$ block $R_1$ of $A$. Let $R_2$ be the lower right $(d-r)\times(d-r)$ block. By applying Laplace expansion to the determinant of $A$ we see that $f_F$, which is the sum of terms of $f$ that have maximal degree in the $x_i$, $i\in S$, is exactly the determinant of the block matrix $$\begin{pmatrix} R_1|_{x_j=0,\, j\not\in S} &0\\ 0& R_2\end{pmatrix}.$$This shows that $f_F$ has a determinantal representation.
\end{proof}

\begin{lem}\label{lem:gotofacedetrep}
 Let $f\in\RR[x_1,\ldots,x_n]$ be homogeneous. Let $F$ be a face of $\newt(f)$. If $f$ has a determinantal representation, then $f_F$ has a determinantal representation as well.
\end{lem}

\begin{proof}
 As in the proof of \Cref{derspec} we can find a multiaffine homogeneous polynomial $p\in\RR[x_{11},\ldots,x_{1r_1},\ldots,x_{nr_n}]$ which has a determinantal representation such that $f=p|_{x_{ij}=x_i}$. In particular, the polytope $\newt(f)$ is the image of $\newt(p)$ under a linear map. There is a face $F'$ of $\newt(p)$ such that $f_F=(p_{F'})|_{x_{ij}=x_i}$. Thus the claim follows from \Cref{lem:gotofacemuaff}.
\end{proof}

\begin{cor}
 Let $f\in\RR[x_1,\ldots,x_n]$ be homogeneous and stable. Let $F$ be a face of $\newt(f)$. If $f$ is weakly determinantal or spectrahedral, then $f_F$ is weakly determinantal resp. spectrahedral as well.
\end{cor}

\begin{proof}
 Let $a\in\RR^n$ and $c\in\RR$ such that $\langle a,p\rangle\geq c$ for all $p\in\newt(f)$ with equality exactly when $p\in F$. If $f$ is spectrahedral, then by \Cref{mariospec} there is a stable polynomial $g$ whose hyperbolicity cone contains the one of $f$ such that $f\cdot g$ has a determinantal representation. Let $F'$ and $F''$ be the faces of $\newt(g)$ and $\newt(f\cdot g)$ where $\langle a, -\rangle$ attains its minimum. Then we have $(f\cdot g)_{F''}=f_F\cdot g_{F'}$ and this polynomial has a determinantal representation by \Cref{lem:gotofacedetrep}. Writing $f_F$ and $g_{F'}$ as limits as in the proof of \Cref{prop:facialstable}, it follows from \Cref{limitcont} that the hyperbolicity cone of $g_{F'}$ contains the one of $f_F$. Thus we conclude that $f_F$ is spectrahedral. If $f$ is weakly determinantal, then we can proceed in the same way by letting $g=f^r$ for some $r\geq0$.
\end{proof}

\section{Classification of Matroids on Eight Elements with Respect to the Half-plane Property}\label{sec:classi}
In this section we classify all matroids on a ground set $E$ with $|E|\leq 8$ that have the half-plane property. Matroids with the half-plane property on at most seven elements were classified in \cite{Choe_2004}, so it remains to consider ground sets of exactly eight elements.

Let $M$ be a matroid on the ground set $\left[ n \right]$ with basis generating polynomial $h_{M}$. 
Recall from \Cref{thm:rayleighbrand} that $M$ has the half-plane property if and only if $\Delta_{ij}h_M(x)\geq0$ for all $x\in\RR^n$ and $1\leq i,j\leq n$. We will make use of this criterion. Wagner and Wei proved that when all proper minors of $M$ have the half-plane property, it suffices to find \emph{one} Rayleigh difference which is nonnegative.

\begin{thm}[Thm.~3 in \cite{Wagner_2009}]\label{thm:wwcrit}
 Let $M$ be a matroid all of whose proper minors have the half-plane property. If there exist distinct indices $i,j\in[n]$ such that $\Delta_{ij}h_M(x)\geq0$ for all $x\in\RR^n$, then $M$ has the half-plane property.
\end{thm}

\Cref{thm:wwcrit} will be the main criterion that we will use to prove that a certain matroid has the half-plane property. But first we narrow down the set of matroids that have to be considered. By \cite[Prop.~4.2]{Choe_2004} having the half-plane property is closed under taking the dual of the matroid. Thus it is enough to consider matroids on eight elements with rank at most four.
Further the half-plane property is preserved by direct sums, by adjoining loops or parallel elements \cite[\S 4]{Choe_2004}.
Therefore, we can restrict attention to simple connected matroids on eight elements of rank at most four. Finally, all rank-1 and rank-2 matroids have the half-plane property by \cite[Cor.~5.5]{Choe_2004}. Thus we can restrict attention to matroids on eight elements of rank three and four.

We will use the Macaulay2 \cite{M2} package ``Matroids'' by Chen \cite{MatroidsSource, MatroidsArticle} where a list of all matroids on at most eight elements is implemented: The command \texttt{allMatroids(8)} yields a list of all matroids on $8$ elements. In the following, we will denote by $\mathcal{M}_i$ the $i$th element of this list (counting starts from zero).
We obtain that the number of non-isomorphic matroids on $8$ elements with rank $3$ or $4$ is $1265$. While $685$ of them are simple, $659$ of them are simple and connected. In \cite{Choe_2004} it was shown that all matroids on at most seven elements have the half-plane property, except for the Fano matroid $F_7$, Non-Fano matroid $ F_7^{-}$, $F_7^{--}$, $F_7^{-3}$, $M(K_{4})+e$ and their duals.
Since the half-plane property is minor-closed, we can exclude all matroids that have one of these matroids as a minor. Then we are left with $309$ simple and connected matroids all of whose proper minors have the half-plane property. By \Cref{thm:wwcrit}, for each of these $309$ matroids, if we can prove that one Rayleigh difference is nonnegative, the matroid in question has the half-plane property.

For doing this, we run a sums of squares test using the Macaulay2 package ``SumsOfSquares'' by Cifuentes, Kahle, Parrilo, and Peyrl \cite{SumsOfSquaresSource,SumsOfSquaresArticle} on all Rayleigh differences of the basis generating polynomials of all these $309$ matroids. 
We obtain that for $287$ matroids from this list there exists some indices $i,j$ such that $\Delta_{ij}(h_{M})$ is a sum of squares.
Among them $256$ matroids are SOS-Rayleigh, i.e., for those matroids \emph{all} Rayleigh differences are a sum of squares. Moreover, for each of the $287$ matroids, there were some indices $i,j$ for which the SDP solver could find a Gram matrix with \emph{rational} entries. Thus we even obtain a \emph{symbolic} certificate for being a sums of squares, and thus for the corresponding matroid to have the half-plane property.
For more information on the source code of the tests and symbolic certificates, please see the ancillary files \texttt{Appendix.pdf, M2Code.m2} and \texttt{SageCode.sage}. For more information on SDP solvers and the ``SumsOfSquares'' package in Macaulay2, we refer to \cite{2012-siamSDP, SumsOfSquaresSource, SumsOfSquaresArticle}.

We thus note that only $22$ matroids are left for which we do not know whether they have the half-plane property: the fact that for those $22$ matroids there are no $i,j$ for which the Rayleigh difference is a sum of squares does not imply that these matroids do not have the half-plane property as there are nonnegative polynomials that are not a sum of squares.  All these $22$ matroids have rank $4$. This implies that among the matroids on eight elements with rank $3$ there are no new forbidden minors for the half-plane property.

In order to investigate the half-plane property of those remaining $22$ matroids, recall that the homogeneous polynomial $h_{M}$ has the half-plane property if and only if it is hyperbolic with respect to $e=(1,\ldots,1)$ and $\RR^{n}_{\geq 0}\subseteq C_{h_{M}}(e)$. Since the roots of a polynomial depend continuously on their coefficients, finding points $e,v\in\mathbb{R}_{\ge0}^n$ for which $h_M(et-v)$ has not only real zeros proves that $h_M$ does not have the half-plane property. Thus, testing whether $h_{M}(et-v)$ has only real roots for $e, v \in \left\{ 0,1 \right\}^{8}$ provides a quick way to eliminate some more matroids. We obtain that $15$ of the $22$ matroids do not pass this test, and the remaining $7$ matroids require more tests to prove or disprove that they have the half-plane property. For the list of these $15$ matroids whose basis generating polynomials fail this test and for the corresponding points $e,v$, we refer to Table~\ref{directions}.

\begin{table}\tiny{
  $  \begin{array}{|c|c|c|}
    \hline
    
      & \textbf{Number of pairs }\mathbf{(e,v)} \textbf{ for which}& \\
    \textbf{ Matroids }& \mathbf{h_{M}(et-v)}&\textbf{One of the pairs }\mathbf{ (e,v)}\\
      &\textbf{has non-real roots}& \\
    
    \hline
    
    \mathcal{M}_{435}&  4& \left(\left( 1, 1, 1, 1, 1, 1, 0, 1 \right), \left(1, 1, 0, 0, 0, 0, 1, 1\right)\right)\\
    
    \hline
    
    \mathcal{M}_{437}&  8&\left( \left( 0, 0, 0, 0, 1, 1, 1, 1 \right),\left( 1, 1, 1, 1, 1, 1, 0, 1 \right) \right)\\
    
    \hline
    
    \mathcal{M}_{439}&  4&\left( \left(1, 0, 1, 1, 1, 1, 1, 1  \right),\left(1, 1, 1, 1, 0, 0, 0, 0  \right) \right)\\
    
    \hline
    
    \mathcal{M}_{443}&  2&\left( \left(1, 1, 1, 1, 0, 1, 1, 1  \right),\left(0, 0, 0, 0, 1, 1, 1, 1  \right) \right)\\
    
    \hline
    
    \mathcal{M}_{450}&  4&\left( \left(0, 1, 1, 1, 1, 1, 1, 1  \right),\left(1, 1, 0, 0, 1, 1, 0, 0  \right) \right)\\
    
    \hline
    
    \mathcal{M}_{455}&  2&\left( \left(1, 1, 0, 1, 0, 1, 0, 1  \right),\left(0, 1, 1, 0, 1, 0, 1, 1  \right) \right)\\
    
    \hline
    
    \mathcal{M}_{460}&  2&\left( \left(0, 0, 1, 1, 1, 1, 1, 1  \right),\left(1, 1, 0, 0, 0, 0, 1, 1  \right) \right)\\
    
    \hline
    
    \mathcal{M}_{461}&  12&\left( \left(1, 1, 1, 1, 0, 1, 1, 1  \right),\left( 0, 0, 1, 1, 1, 1, 0, 0 \right) \right)\\
    
    \hline
    
    \mathcal{M}_{465}&  10&\left( \left(0,1, 1, 1, 1, 1, 1, 1   \right),\left(1, 1, 0, 0, 0, 0, 1, 1  \right) \right)\\
    
    \hline
    
    \mathcal{M}_{466}&  38&\left( \left(0, 0, 1, 1, 1, 1, 0, 1  \right),\left(1, 1, 0, 0, 0, 0, 1, 1  \right) \right)\\
    
    \hline
    
    \mathcal{M}_{467}&  62&\left( \left(1, 1, 0, 1, 1, 0, 1, 1  \right),\left( 0, 0, 1, 1, 1, 1, 0, 0 \right) \right)\\
    
    \hline
    
    \mathcal{M}_{548}&  6&\left( \left(1, 1, 1, 1, 0, 0, 0, 1  \right),\left( 0, 0, 0, 0, 1, 1, 1, 1 \right) \right)\\
    
    \hline
    
    \mathcal{M}_{549}&  4&\left( \left(1, 1, 1, 1, 0, 0, 1, 1  \right),\left( 1, 1, 1, 1, 0, 0, 1, 1 \right) \right)\\
    
    \hline
    
    \mathcal{M}_{570}&  78&\left( \left(1, 1, 0, 0, 1, 1, 1, 1  \right),\left( 0, 0, 1, 1, 1, 1, 0, 0 \right) \right)\\
    
    \hline
    
    \mathcal{M}_{575}&  158&\left( \left(1, 1, 1, 1, 1, 1, 0, 1  \right),\left( 0, 0, 0, 1, 1, 0, 1, 1 \right) \right)\\
    
    \hline
  \end{array}$
  \caption{$15$ matroids and a sample of directions for which they fail the hyperbolicity test.}
  \label{directions}}
\end{table}

For each of the remaining $7$ matroids $M$,  using the Julia package ``HomotopyContinuation.jl'' by Breiding and Timme \cite{inbook}, we found a point $x\in \RR^{6}$ for which $\Delta_{67}(h_{M})(x)< 0$ which confirms that they do not have the half-plane property as well. More precisely, we computed all critical points of the Rayleigh difference and plugged in nearby rational points. For the list of points where the Rayleigh difference is negative, see Table~\ref{negativepts}.
This gives us the following main result of this section.

\begin{table}{
  $ \begin{array}{|c|c|}
    \hline
    
    \textbf{ Matroids}& \mathbf{x\in\RR^{6} }\textbf{ s.t. }\mathbf{ \Delta_{6,7}(h_{M})(x)<0}\\
    
    \hline
    \mathcal{M}_{424} & \left(4,30,1,7,-32,-4\right)\\
    \hline
    \mathcal{M}_{430} & \left(80,19,-31,-31,-17,-4\right)\\
    \hline
    \mathcal{M}_{431} & \left(60,27,-90,-22,27,5\right)\\
    \hline
    \mathcal{M}_{436} & \left(40,309,-40,-306,9,73\right)\\
    \hline
    \mathcal{M}_{462} & \left(20,55,-11,-4,-52,-19\right)\\
    \hline
    \mathcal{M}_{463} & \left(30,399,-111,-10,-368,-28\right)\\
    \hline
    \mathcal{M}_{550} & \left(50,-25,94,-45,-142,66\right)\\
    \hline
  \end{array}$
  \caption{Seven matroids and  points $x\in \RR^{6}$ at which one of their Rayleigh differences is negative.}
  \label{negativepts}}
\end{table}
\begin{thm}\label{thm:forbmin}
 There are exactly $22$ matroids on $8$ elements all of whose minors have the half-plane property but which do not have the half-plane property themselves. These $22$ matroids are the matroids $\mathcal{M}_{k}$ for $k$ in the following list:
\begin{align*}
    & \left\{ 424, 430, 431, 435, 436, 437, 439, 443, 450, 455, 460, 461, 462, 463, 465,\right.\\
    &\left. 466, 467, 548, 549, 550, 570, 575\right\}.
\end{align*}
All of them are sparse paving matroids of rank $4$. Furthermore, all of these matroids are self-dual except for the following dual pairs: $\mathcal{M}_{424}=\mathcal{M}_{436}^*$, $\mathcal{M}_{430}=\mathcal{M}_{548}^*$, $\mathcal{M}_{450}=\mathcal{M}_{549}^*$ and $\mathcal{M}_{455}=\mathcal{M}_{550}^*$.
\end{thm}
We now take a closer look at some of the $22$ matroids from \Cref{thm:forbmin}. In the following, we will mostly stick to the catalog of \cite{oxley} regarding special matroids. Further $\Ext(M)$ of a matroid $M$ refers to the free extension of $M$ by an element $e$, see \cite[page 270]{oxley}. The co-extension $\CoExt(M)$ of $M$ refers to the dual of the extension of $M^{*}$, such that $\CoExt/e=M$. Note that $\Ext(M)$ and $\CoExt(M)$ can be obtained by using the SageMath commands \texttt{M.extension()} and \texttt{M.coextension()} respectively.

\begin{ex}
 The matroid $\mathcal{M}_{430}$ is the co-extension $\CoExt(P_7)$ of the ternary 3-spike $P_7$. Both are depicted in \Cref{fig:p7endcoext}. Further minors of $\mathcal{M}_{430}$ include $P_{6}$, $Q_{6}$, $R_{6}$, $\CoExt(R_{6})$ and $\Ext(Q_{6})$ which all have the half-plane property. Its dual is the matroid $\mathcal{M}_{548}$.
\end{ex}

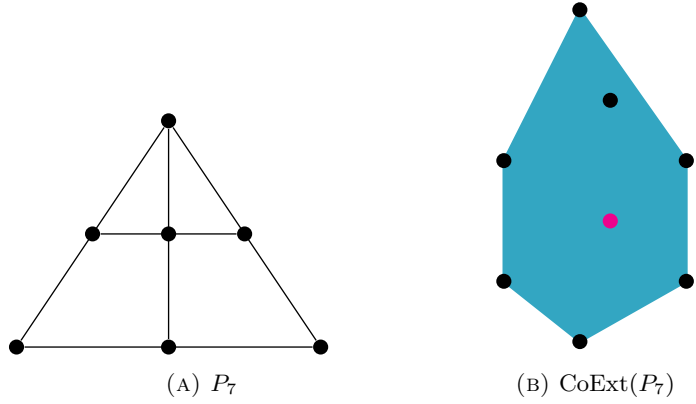
\begin{figure}
  \begin{center}
    \begin{subfigure}[b]{0.4\linewidth}
      \begin{tikzpicture}[every node/.style={fill=black, inner sep=2pt, circle}]
        \node(a) at (0,0){};
        \node(b) at (2,3){};
        \node(c) at (4,0){};
        \node (ab) at ($(a)!0.5!(b)$) {};
        \node (bc) at  ($(b)!0.5!(c)$){};
        \node (ac) at  ($(a)!0.5!(c)$){};
        \node (abbc) at  ($(ab)!0.5!(bc)$){};
        \draw [line width=0.5pt](a)--(b)--(c)--(a);
        \draw [line width=0.5pt](ab)--(bc);
        \draw[line width=0.5pt] (b)--(ac);
      \end{tikzpicture}
      \caption{$P_{7}$}
    \end{subfigure}
    \begin{subfigure}[b]{0.4\linewidth}
      \definecolor{qqqqff}{rgb}{0.20,0.65,0.76}
      \definecolor{ffvvqq}{rgb}{0.17,0.70,0.76}
      \begin{center}
         \begin{tikzpicture}[x=0.4cm, y=0.4cm, line width=1pt, every node/.style={fill=black, inner sep=2pt, circle}]
         \coordinate (ca) at (0,0){};
         \coordinate(cb) at (2.5,-2){};
         \coordinate (cc) at (6,0){};
         \coordinate[color=magenta](cd) at (3.5,2){};
         \coordinate(ce) at (6,4){};
         \coordinate (cf) at (3.5,6){};
         \coordinate(cg) at (0,4){};
         \coordinate (ch) at (2.5,9){};
         
        \draw[fill=qqqqff, color=qqqqff, opacity=0.3] (ca)--(cb)--(cc)--(cd)--(ca);
        \draw[fill=qqqqff, color=qqqqff, opacity=0.3](cc)--(cd)--(cf)--(ce)--(cc);
        \draw[fill=qqqqff, color=qqqqff, opacity=0.3](cb)--(ch)--(cf)--(cd)--(cb);
        \draw[fill=qqqqff, color=qqqqff, opacity=0.3](ca)--(cd)--(cf)--(cg)--(ca);
        \draw[ fill=qqqqff, color=qqqqff, opacity=0.3](cg)--(cd)--(ce)--(ch)--(cg);
        
       \node (a) at (0,0){};
        \node (b) at (2.5,-2){};
        \node (c) at (6,0){};
        \node[color=magenta](d) at (3.5,2){};
        \node(e) at (6,4){};
        \node (f) at (3.5,6){};
        \node(g) at (0,4){};
        \node (h) at (2.5,9){};
       \end{tikzpicture}
       \end{center}
      \caption{CoExt($P_{7}$)}
    \end{subfigure}
  \end{center}
  \caption{The matroids $P_{7}$ and $\mathcal{M}_{430}=\CoExt(P_7)$.}\label{fig:p7endcoext}
\end{figure}

\begin{ex}\label{ex:p8ppp}
 The rank-$4$ matroid $P_8$ is represented over $\QQ$ by the matrix:
\[
 \begin{pmatrix}
  1&0&0&0&0&1&1&2 \\
  0&1&0&0&1&0&1&1 \\
  0&0&1&0&1&1&0&1 \\
  0&0&0&1&2&1&1&0 \\
 \end{pmatrix}
\]
 The circuit hyperplanes of $P_8$ are $\{1,2,3,8\}$, $\{1,2,4,7\}$, $\{1,3,4,6\}$, $\{2,3,4,5\}$, $\{1,4,5,8\}$, $\{2,3,6,7\}$, $\{1,5,6,7\}$, $\{2,5,6,8\}$, $\{3,5,7,8\}$ and $\{4,6,7,8\}$.
 We denote by $P_8'$, $P_8''$ and $P_8'''$ the matroids obtained by relaxing the first, the first and the second, resp. all three of the following circuit hyperplanes: $\{1,4,5,8\}$, $\{2,3,6,7\}$ and $\{4,6,7,8\}$. It was shown in \cite[Ex.~11.8]{Choe_2004} that $P_8$, $P_8'$ and $P_8''$ do not have the half-plane property. In our list, these are the matroids $\mathcal{M}_{575}$, $\mathcal{M}_{570}$ and $\mathcal{M}_{467}$. Furthermore, the matroid $\mathcal{M}_{466}$ is the relaxation ${P_{8}^{\prime\prime\prime}}$ which does not have the half-plane property as well. In \cite[Prop.~4]{matroids9} the matroid ${P_{8}^{\prime\prime\prime}}$ is denoted by $P_3$ and it is shown that it is not representable over any field. Further note that in \cite{oxley} the matroid $P_8''$ is denoted by $P_8^=$. It is an excluded minor for being representable over the field of $4$ elements.
\end{ex}

\begin{ex}
 The matroid $\mathcal{M}_{431}$ has only four circuit hyperplanes: $\{1, 2, 3, 4\}$, $\{1, 2, 5, 6\}$, $\{1, 3, 5, 7\}$ and $\{2, 3, 5, 8\}$.
\end{ex}

Any matroid having one of our $22$ matroids as a minor does not have the half-plane property.

\begin{ex} 
 The \emph{Extended Ternary Golay code} is the matroid represented by the following generating matrix of the Extended Ternary Golay code over the field $\mathbb{F}_{3}$
   \[\left(\begin{array}{rrrrrrrrrrrr}
       1 & 0 & 0 & 0 & 0 & 0 & 1 & 1 & 1 & 2 & 2 & 0 \\
       0 & 1 & 0 & 0 & 0 & 0 & 1 & 1 & 2 & 1 & 0 & 2 \\
       0 & 0 & 1 & 0 & 0 & 0 & 1 & 2 & 1 & 0 & 1 & 2 \\
       0 & 0 & 0 & 1 & 0 & 0 & 1 & 2 & 0 & 1 & 2 & 1 \\
       0 & 0 & 0 & 0 & 1 & 0 & 1 & 0 & 2 & 2 & 1 & 1 \\
       0 & 0 & 0 & 0 & 0 & 1 & 0 & 1 & 1 & 1 & 1 & 1
     \end{array}\right).
  \]
 The supports of the codewords form a Steiner System $S(5,6,12)$, i.e., a $12$ element set with a collection $\mathcal{D}$ of $6$ element subsets (called blocks) such that every $5$ element subset of $S$ is contained in exactly one block. Here, a ground set $E$ with $12$ elements with the collection $\mathcal{D}$ of hyperplanes define a matroid. It does not have the half-plane property since it has $P_8$ as a minor.
\end{ex}

Further matroids having one of our $22$ matroids as a minor include the co-extensions $\CoExt(P_{8})$, $\CoExt(J)$, $\CoExt(T_{8})$, $\CoExt(S_{8})$, $\CoExt(N1)$, $\CoExt(N2)$, $\CoExt(AG(2,3))$, the extensions  $\Ext(P_{8})$, $\Ext(T_{8})$, $\Ext(J)$, $\Ext(N1)$ and the matroid $AG(3,3)$. These all do not have the half-plane property.

\begin{ex}
 The Pappus and the non-Pappus matroid are of rank $3$ on $9$ elements \cite[Ex.~1.5.15]{oxley}. They do not have the half-plane property by \cite[Ex.~11.9]{Choe_2004}. Their only two non-isomorphic deletions were proven to have the half-plane property in \cite{Wagner_2009}. Thus these two matroids are minor-minimal with respect to the half-plane property. In \cite[Ex.~11.10]{Choe_2004} it is shown that the matroid $(\textnormal{non-Pappus}\setminus 9)+e$ obtained by adding a new element freely to the contraction of the non-Pappus matroid by an element contained in two circuits does not have the half-plane property either. The authors of \cite{Choe_2004} suspect that this matroid is also minor-minimal with respect to the half-plane property. This is indeed the case: By \Cref{thm:forbmin} there are no minor-minimal matroids of rank $3$ on $8$ elements. Thus minor-minimality follows from the fact that none of the five matroids of rank $3$ on $7$ elements that do not have the half-plane property is a minor of $(\textnormal{non-Pappus}\setminus 9)+e$. In Section~\ref{9}, we confirm that they are the only minimal forbidden minors for the half-plane property of rank $3$ on $9$ elements.
\end{ex}

\begin{rem}
 A binary matroid has the half-plane property if and only if it is a regular matroid and, similarly, a ternary matroid has the half-plane property if and only if it is a $\sqrt[6]{1}$-matroid \cite[\S15.8]{oxley}. In particular, the list of excluded minors for the classes of binary and ternary matroids with the half-plane property is finite. In this context Oxley poses the problem of finding all excluded minors for the class of quaternary matroids with the half-plane property \cite[Prob.~15.8.8]{oxley}. In addition to the excluded minors for quaternary matroids, see \cite[Thm.~6.5.9]{oxley}, this list contains the Fano matroid $F_7$, its relaxation $F_7^{--}$, their duals and the matroid $
 P_8'$ from \Cref{ex:p8ppp}. Any quaternary matroid on at most $8$ elements that does not have the half-plane property has one of these $5$ matroids as a minor. Further excluded minors for the class of quaternary matroids with the half-plane property are the Pappus matroid and its dual.
\end{rem}

\begin{rem}
  A matroid $M$ with ground set $E$ is called \emph{Rayleigh}, if $\Delta_{ij}(h_{M})(x)\geq0$ for all $i,j\in E$ and all $x$ in the \emph{nonnegative orthant}. It is clear from \Cref{thm:rayleighbrand} that the half-plane property implies Rayleigh. The converse is not true. For instance every matroid of rank $3$ is Rayleigh \cite{MR2134193}. When searching for critical points of Rayleigh differences using the Julia package ``HomotopyContinuation.jl'', we found that the matroids $\mathcal{M}_k$ for $k$ from the following list are not Rayleigh: $$\{768,816,821,825,878,879,882,883,891,894,895,896,910,911,912\}.$$Note that the matroid $\mathcal{M}_{912}$ is the matroid $S_8$ \cite[p.~648]{oxley}.
  
 Furthermore, for fixed indices $i,j$, we can express $h_{M}$ as $h_{M}=ax_{i}x_{j}+bx_{i}+cx_{j}+d$ where $a,b,c,d\in\RR[x_{1},\dots,x_{n}]$ are polynomials that do not depend on $x_{i}$ and $x_{j}$. Then we have $\Delta_{i,j}h_{M}=ad-bc$, and the supremum of the ratio $\frac{bc}{ad}(x)$ over all $x$ in the nonnegative orthant and all $i,j$ is called the \emph{correlation constant} of the matroid $M$. If $M$ is not Rayleigh, then the correlation constant is larger than $1$. It was conjectured in \cite{Huh-21} that the correlation constant cannot exceed $\frac{8}{7}$. Our computations support their conjecture, see also the ancillary file \texttt{Appendix.pdf}.
\end{rem}

\subsection{SOS-Rayleigh and weakly determinantal matroids}\label{sosdet}
In a celebrated article Br\"and\'en \cite{Branden2011} disproved a predecessor of the generalized Lax Conjecture by showing that the V\'amos matroid $V_8$ is not weakly determinantal although it has the half-plane property \cite{Wagner_2009}. Further such matroids were constructed in \cite{Amini2018,burton2014real}. In this section we will add some further examples to this list. The examples from the literature rely on the following criterion:

\begin{thm}[Br\"and\'en \cite{Branden2011}]
 Let $M$ be a weakly determinantal matroid on the ground set $[n]$. Then its rank function $\rank$ satisfies the \emph{Ingleton inequalities}: \begin{align*}
    & \rank(P_{1}\cup P_{2})+ \rank(P_{1} \cup P_{3})+ \rank(P_{1}\cup P_{4})+\rank(P_{2}\cup P_{3})+\rank(P_{2}\cup P_{4})\\
  \geq & \rank(P_{1})+ \rank(P_{2})+ \rank(P_{1}\cup P_{2}\cup P_{3})+ \rank(P_{1}\cup P_{2} \cup P_{4})+\rank(P_{3} \cup P_{4})
\end{align*} 
 for all $P_1,P_2,P_3,P_4\subseteq[n]$.
\end{thm}

For the V\'amos matroid there exist disjoint subsets $P_1,P_2,P_3,P_4\subseteq[8]$ of size $2$ that violate this inequality. This motivates the following definition.

\begin{Def}
  A sparse paving matroid $M$ on a ground set $E$ of rank $r$ is called \emph{V\'amos-like} if there exist pairwise disjoint sets $P_{1},P_{2}, P_{3},P_{4},K$ where $|P_{1}|=|P_{2}|=|P_{3}|=|P_{4}|=2$ and $|K|=r-4$ such that $K\cup P_{i}\cup P_{j}$ for $i<j, (i,j)\neq (3,4)$ are non-bases, while $K\cup P_{3} \cup P_{4}$ is a basis.  
\end{Def}

In \cite{ingleton} it was shown that V\'amos-like matroids violate the Ingleton inequalities. Mayhew and Royle show that among the matroids on eight elements \cite{matroids9}, there are $44$ non-representable matroids and $39$ of them are V\'amos-like. The remaining $5$ matroids are four relaxations of $P_{8}$ and a relaxation of $L_{8}$. Our tests concluded that out of these $44$ matroids, which are not representable over any field, only the V\'amos matroid itself has the half-plane property. Furthermore, the minor-minimal matroids from \Cref{thm:forbmin} all satisfy Ingleton's inequalities.

For finding further instances of matroids with the half-plane property that are not weakly determinantal, we use the following criterion.

\begin{thm}[Cor.~4.3 in \cite{Kummer-2013-hyperbolic-polynomials-interlacers}]\label{sosR-weak}
  If a homogeneous, multiaffine polynomial is weakly-determinantal, then it is SOS-Rayleigh.
\end{thm}

Indeed, we found $14$ matroids, including the V\'amos matroid $\mathcal{M}_{502}$, for which some Rayleigh differences are a sum of squares but not all. This shows that these have the half-plane property but are not weakly determinantal.

\begin{thm}
 The matroids $\mathcal{M}_{k}$ for $k$ in
\[
  \left\{ 409, 413, 414, 415, 417, 418, 419, 438, 440, 445, 498, 500, 501, 502 \right\}
\] have the half-plane property but are not SOS-Rayleigh. In particular, these matroids are not weakly determinantal. 
\end{thm}
For each of the matroids $M$ in the list, we produce symbolic certificates for the existence of a pair $i,j$ such that $\Delta_{i,j}h_{M}$ is not an sum of squares by using the following argument, and Lemma~\ref{Gram}. Code for producing the certificates can be found in the ancillary files \texttt{M2Code.m2} and \texttt{Appendix.pdf}.

A  homogeneous polynomial $h\in\RR[x_{1},\dots,x_{n}]$ of degree $2d-2$ is a sum of squares if there exists a symmetric positive semidefinite matrix $G\in\RR^{\binom{n}{d-1}\times\binom{n}{d-1}}$ (Gram matrix) such that $h=m^{t}Gm$, where $m\in\RR^{\binom{n}{d-1}} $ is the vector of all multiaffine monomials of degree $d-1$. Therefore, in order to prove that $h$ is not a sum of squares, one needs to show that, the Gram spectrahedron
\[S_{G}:=\left\{ \lambda\in\RR^{k}: G_{0}+\lambda_{1}G_{1}+\dots \lambda_{k}G_{k}\succeq 0\right\}
\]
is empty.
\begin{lem} \label{Gram} Let $h\in\RR[x_{1},\dots x_{n}]$ be a non-zero homogeneous polynomial of degree $2d-2$, and $S_{G}$ be its Gram spectrahedron defined by the pencil
  \[
    G_{0}+\lambda_{1}G_{1}+\dots \lambda_{k}G_{k}\succeq 0
  \] for some real symmetric matrices $G_{1},\dots,G_{k}\in\RR^{\binom{n}{d-1}\times \binom{n}{d-1}}$.
If there exists a positive definite matrix $A\in\RR^{\binom{n}{d-1}}$ such that $\tr(AG_{i})=0$ for all $i=0,\dots,k$, then $S_{G}$ is empty. 
 \end{lem}
 \begin{proof}
   Assume that $S_{G}\neq \emptyset$. Then there exists $\lambda\in\RR^{k}$ such that $$G=G_{0}+\lambda_{1}G_{1}+\dots \lambda_{k}G_{k}$$ is positive semidefinite with $\tr(AG)=0$. Since $A$ is positive definite,
   there exist an invertible matrix $B\in\RR^{\binom{n}{d-1}\times \binom{n}{d-1}}$ such that
   $BB^{t}=A$. Since $G$ is positive semidefinite and non-zero, 
   \[0=\tr(AG)=\tr(BB^{t}G)=\tr(B^{t}GB)>0\]
   gives a contradiction.
 \end{proof}

\begin{ex}
 The circuit hyperplanes of the sparse paving matroid $\mathcal{M}_{417}$ are given by the quadrilateral faces of the heptahedron depicted in \Cref{fig:hepta}.
\end{ex}

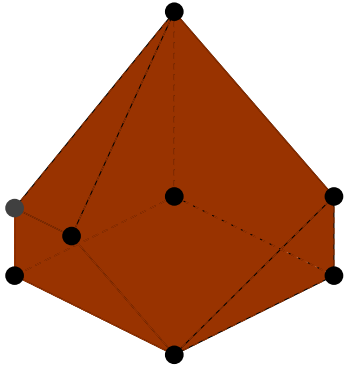
\begin{figure}
 \begin{center}
\definecolor{zzttqq}{rgb}{0.6,0.2,0}
\definecolor{uququq}{rgb}{0.25,0.25,0.25}
\begin{tikzpicture}[line cap=round,line join=round,>=triangle 45,x=0.7cm,y=0.7cm]
\clip(-2,2) rectangle (5,9.5);
\fill[dash pattern=on 2pt off 2pt,color=zzttqq,fill=zzttqq,fill opacity=0.1] (1.5,9) -- (-0.43,4.75) -- (1.5,2.5) -- (4.5,5.5) -- cycle;
\fill[color=zzttqq,fill=zzttqq,fill opacity=0.1] (-1.5,5.28) -- (-0.43,4.75) -- (1.5,2.5) -- (-1.5,4) -- cycle;
\fill[dash pattern=on 2pt off 2pt,color=zzttqq,fill=zzttqq,fill opacity=0.1] (1.5,2.5) -- (4.5,4) -- (1.5,5.5) -- (-1.5,4) -- cycle;
\fill[dash pattern=on 2pt off 2pt,color=zzttqq,fill=zzttqq,fill opacity=0.1] (-1.5,4) -- (-1.5,5.28) -- (1.5,9) -- (1.5,5.5) -- cycle;
\fill[dash pattern=on 2pt off 2pt,color=zzttqq,fill=zzttqq,fill opacity=0.1] (1.5,9) -- (4.5,5.5) -- (4.5,4) -- (1.5,5.5) -- cycle;
\draw (1.5,2.5)-- (4.5,4);
\draw (1.5,2.5)-- (-1.5,4);
\draw [dash pattern=on 2pt off 2pt] (1.5,5.5)-- (-1.5,4);
\draw [dash pattern=on 2pt off 2pt] (1.5,5.5)-- (4.5,4);
\draw [dash pattern=on 2pt off 2pt] (1.5,5.5)-- (1.5,9);
\draw (1.5,2.5)-- (4.5,5.5);
\draw (4.5,5.5)-- (1.5,9);
\draw (4.5,5.5)-- (4.5,4);
\draw (-0.43,4.75)-- (1.5,2.5);
\draw (-0.43,4.75)-- (1.5,9);
\draw (1.5,9)-- (-1.5,5.28);
\draw (-1.5,5.28)-- (-0.43,4.75);
\draw (-1.5,5.28)-- (-1.5,4);
\draw [dash pattern=on 2pt off 2pt,color=zzttqq] (1.5,9)-- (-0.43,4.75);
\draw [dash pattern=on 2pt off 2pt,color=zzttqq] (-0.43,4.75)-- (1.5,2.5);
\draw [dash pattern=on 2pt off 2pt,color=zzttqq] (1.5,2.5)-- (4.5,5.5);
\draw [dash pattern=on 2pt off 2pt,color=zzttqq] (4.5,5.5)-- (1.5,9);
\draw [color=zzttqq] (-1.5,5.28)-- (-0.43,4.75);
\draw [color=zzttqq] (-0.43,4.75)-- (1.5,2.5);
\draw [color=zzttqq] (1.5,2.5)-- (-1.5,4);
\draw [color=zzttqq] (-1.5,4)-- (-1.5,5.28);
\draw [dash pattern=on 2pt off 2pt,color=zzttqq] (1.5,2.5)-- (4.5,4);
\draw [dash pattern=on 2pt off 2pt,color=zzttqq] (4.5,4)-- (1.5,5.5);
\draw [dash pattern=on 2pt off 2pt,color=zzttqq] (1.5,5.5)-- (-1.5,4);
\draw [dash pattern=on 2pt off 2pt,color=zzttqq] (-1.5,4)-- (1.5,2.5);
\draw [dash pattern=on 2pt off 2pt,color=zzttqq] (-1.5,4)-- (-1.5,5.28);
\draw [dash pattern=on 2pt off 2pt,color=zzttqq] (-1.5,5.28)-- (1.5,9);
\draw [dash pattern=on 2pt off 2pt,color=zzttqq] (1.5,9)-- (1.5,5.5);
\draw [dash pattern=on 2pt off 2pt,color=zzttqq] (1.5,5.5)-- (-1.5,4);
\draw [dash pattern=on 2pt off 2pt,color=zzttqq] (1.5,9)-- (4.5,5.5);
\draw [dash pattern=on 2pt off 2pt,color=zzttqq] (4.5,5.5)-- (4.5,4);
\draw [dash pattern=on 2pt off 2pt,color=zzttqq] (4.5,4)-- (1.5,5.5);
\draw [dash pattern=on 2pt off 2pt,color=zzttqq] (1.5,5.5)-- (1.5,9);
\begin{scriptsize}
\fill [color=black] (1.5,2.5) circle (3.5pt);
\fill [color=black] (4.5,4) circle (3.5pt);
\fill [color=black] (-1.5,4) circle (3.5pt);
\fill [color=black] (1.5,5.5) circle (3.5pt);
\fill [color=black] (1.5,9) circle (3.5pt);
\fill [color=black] (4.5,5.5) circle (3.5pt);
\fill [color=black] (-0.43,4.75) circle (3.5pt);
\fill [color=uququq] (-1.5,5.28) circle (3.5pt);
\end{scriptsize}
\end{tikzpicture}

 \end{center}
\caption{The dependent sets of size $4$ of the matroid $\mathcal{M}_{417}$ are the five quadrilateral facets of the depicted heptahedron.}\label{fig:hepta}
\end{figure}

\begin{rem}
 For the matroids $\mathcal{M}_{k}$ where $k$ is in
\[
  \left\{ 393, 395, 397, 399, 400, 401, 403, 404, 405, 410, 411, 412, 421, 423, 433, 664, 717 \right\}
\]there are some indices  $i,j$ such that the SDP solver cannot find positive semidefinite rational Gram matrices during the SOS test on $\Delta_{i,j}$ but only positive semidefinite Gram matrices having floating points as entries (i.e, they are approximate). These numerical results indicate that these matroids might be SOS-Rayleigh but we do not have a symbolic certificate for that. However, there is always another pair of indices for which we can find a rational positive semidefinite Gram matrix, so this issue does not affect the proof of the half-plane property.
\end{rem}

\begin{table}
  $\begin{array}{|c|c|}
    \hline
    \textbf{Properties}&\textbf{Count}\\
    \hline
    \text{Half-plane property}& 287\\
    \hline
    \text{Certified SOS-Rayleigh} & 256 \\
   \hline
    \text{Half-plane property but not SOS-Rayleigh } & 14\\
    \hline
    \text{Not having the HPP}& 22\\
    \hline
\end{array}$
  \caption{Summary of the classification of simple connected matroids on $8$ elements and rank $3$ or $4$ all of whose proper minors have the half-plane property.}
\end{table}

  We provide the list of simple matroids on $8$ elements of rank $3$ and $4$ that are HPP, that are not HPP, some matroids that are not SOS-Rayleigh, and matroids for which we could not detect whether they are SOS-Rayleigh at  \url{https://zenodo.org/record/6108027} .

\section{{Experiments on Matroids on Nine Elements}}\label{9}
In this section, we implement the methods used in the previous section on simple and connected matroids on $9$ elements of rank $3$ and $4$. We use the list of all simple matroids on 9 elements provided by  Matsumoto, Moriyama, Imai and Bremner in \cite{Matsumoto-2011}. In this case the number of matroids is very high and we provide our results at \url{https://zenodo.org/record/6108027}.

After excluding forbidden minors, we conduct an SOS test on the Rayleigh differences of simple connected matroids on $9$ elements of rank $3$. Our tests show that the only minimal forbidden minors for the half-plane property on $9$ elements of rank $3$ are the Pappus, non-Pappus and (non-Pappus$\backslash 9$)$+e$ matroids. The remaining $116$ matroids have the half-plane property. For ten of these matroids there are some indices $(i,j) $ such that for $\Delta_{i,j}$ the SDP solver can only find positive semidefinite Gram matrices with floating point entries (i.e., they are approximate). On the other hand, we could not certify that they are not SOS-Rayleigh. The numerical data from the SDP outputs suggests that such $\Delta_{i,j}$ are sums of squares with non-rational coefficients only.

  \begin{ex}
The matroid $M$ shown in Figure~\ref{r3nonsos} on ground set $E=[9]$ with non-bases $\{1,2,3\} $, $\left\{1,4,5  \right\} $, $\left\{ 2,4,6 \right\}, \left\{ 3,7,8 \right\} $ and
$ \left\{ 5,7,9 \right\}$ is one of the ten matroids of rank $3$ with the HPP for which our algorithm cannot prove or disprove that they are SOS-Rayleigh.
\end{ex}

\begin{figure}
  \begin{center}
    \begin{subfigure}{0.4\linewidth}
      \begin{center}
    \begin{tikzpicture}[every node/.style={fill=black, inner sep=2pt, circle},x=0.5cm,y=0.5cm]
        \node(a) at (0,0){};
        \node(b) at (1.5,0){};
        \node(c) at (3,0){};
        \node(d) at (3,2.5){};
        \node(e) at (3,5){};
        \node(f) at (0,5){};
        \node(g) at (1.5,5){};
        \node(h) at (0,2.5){}; 
        \node (bh) at ($(b)!0.5!(h)$){};
        \draw [line width=0.5pt](a)--(b)--(c);
        \draw [line width=0.5pt](c)--(d)--(e)--(g)--(f);
          \draw [line width=0.5pt](a)--(h)--(f);
        \draw [line width=0.5pt](b)--(bh)--(h);
      \end{tikzpicture}
      \end{center}
      \caption{A matroid with the HPP for which it is unknown whether it is SOS-Rayleigh}\label{r3nonsos}
    \end{subfigure}
    \begin{subfigure}{0.4\linewidth}
      \begin{center}
      \begin{tikzpicture}[every node/.style={fill=black, inner sep=2pt, circle},x=0.5cm,y=0.5cm]
        \node(a) at (0,0){};
        \node(b) at (2.5,0){};
        \node(c) at (5,0){};
        \node(d) at (0,5){};
        \node(e) at (2.5,5){};
        \node(f) at (5,5){};
        \node[fill=magenta](g) at (4,2.5){};
        
        \node (ae) at ($(a)!0.5!(e)$){};
        \node (af) at ($(a)!0.5!(f)$){};
        \draw [line width=0.5pt](a)--(b)--(c);
        \draw [line width=0.5pt](d)--(e)--(f);
        \draw [line width=0.5pt](a)--(ae)--(e);
        \draw [line width=0.5pt](a)--(af)--(f);
        \draw [line width=0.5pt](d)--(ae)--(b);
         \draw [line width=0.5pt](d)--(af)--(c);
    
       \end{tikzpicture}
       \end{center}
       \caption{(Non-Pappus$\backslash 9$) $+e$}
    \end{subfigure}
  \end{center}
  \caption{}
\end{figure}

There are $185982$ simple matroids on $9$ elements of rank $4$. We first extend the list of excluded minors (by adding those of rank $3$ on $9$ elements, and of rank $4$ on $8$ elements), and eliminate matroids having one of the 35 forbidden minors as a minor. By applying a sums of squares test on the Rayleigh differences of the remaining connected matroids, we found that  $4125$ matroids have the half-plane property. For each of them, there is a pair $(i,j)$ such that the SDP solver could find positive semidefinite rational Gram matrices during the SOS test on $\Delta_{i,j}$. In addition, there are $819$ matroids for which the SDP solvers could only find positive semidefinite Gram matrices with floating point entries (i.e., they are approximate). The numerical results suggest they might have the half-plane property, but our computations do not provide a proof. We list those matroids under the category ``Candidates for having the half-plane property''.

  We found $1218$ matroids that do not have the half-plane property all of whose proper minors have the half-plane property. By using the Julia package ``Homotopy Continuation'',  for each such $M$, we were able to find indices $(i,j)$ and $x\in\RR^{7}$ such that $\Delta_{i,j}(h_{M})(x)<0$.
  \begin{table}
  $\begin{array}{|c|c|}
    \hline
    \textbf{Properties}&\textbf{Count}\\
    \hline
    \text{Half-plane property}& 116\\
   \hline
    \text{Undetected SOS-Rayleigh } & 10\\
    \hline
    \text{Not having the HPP}& 3\\
    \hline
\end{array}$
  \caption{Summary of the classification of simple connected matroids on $9$ elements and rank $3$ all of whose proper minors have the half-plane property.}
\end{table}

\begin{table}
  $\begin{array}{|c|c|}
    \hline
    \textbf{Properties}&\textbf{Count}\\
    \hline
    \text{Half-plane property}& 4125\\
    \hline
    \text{Candidates for having the HPP} & 819 \\
   \hline
    \text{Not having the HPP}& 1218\\
    \hline
    \text{Undetected}& 556\\
    \hline
\end{array}$
  \caption{Summary of the test results on simple connected matroids  on $9$ elements of rank $4$ all of whose proper minors have the half-plane property.}
\end{table}

Note that our experiments do not provide a complete classification with respect to the half plane property for matroids on $9$ elements with rank $4$. There are $556$ matroids that neither pass the SOS test, nor the test for finding negative points. We suspect that the Rayleigh differences of those matroids are non-negative, but not sums of squares. However, our computations do not provide a proof. These are listed under the name ``Undetected''. Due to computational time constraints, we did not perform tests for the SOS-Rayleigh property on matroids on $9$ elements of rank $4$: Once we found a sums of squares representation for one pair of indices, we did not continue to check the remaining ones.

\nocite{sage}

 \bigskip
  \noindent \textbf{Acknowledgements.}
 We thank several anonymous referees for providing helpful comments on the paper.

\bibliographystyle{plain}
\bibliography{biblio.bib}

\def\cprime{$'$}
\begin{thebibliography}{10}

\bibitem{Amini2018}
Nima Amini and Petter Br\"{a}nd\'{e}n.
\newblock Non-representable hyperbolic matroids.
\newblock {\em Adv. Math.}, 334:417--449, 2018.

\bibitem{2012-siamSDP}
Grigoriy Blekherman, Pablo~A. Parrilo, and Rekha~R. Thomas, editors.
\newblock {\em Semidefinite optimization and convex algebraic geometry},
  volume~13 of {\em MOS-SIAM Series on Optimization}.
\newblock Society for Industrial and Applied Mathematics (SIAM), Philadelphia,
  PA; Mathematical Optimization Society, Philadelphia, PA, 2013.

\bibitem{negdep}
Julius Borcea, Petter Br\"{a}nd\'{e}n, and Thomas~M. Liggett.
\newblock Negative dependence and the geometry of polynomials.
\newblock {\em J. Amer. Math. Soc.}, 22(2):521--567, 2009.

\bibitem{Branden2007}
Petter Br{\"{a}}nd\'{e}n.
\newblock Polynomials with the half-plane property and matroid theory.
\newblock {\em Adv. Math.}, 216(1):302--320, 2007.

\bibitem{Branden2011}
Petter Br{\"{a}}nd\'{e}n.
\newblock Obstructions to determinantal representability.
\newblock {\em Adv. Math.}, 226(2):1202--1212, 2011.

\bibitem{inbook}
Paul Breiding and Sascha Timme.
\newblock Homotopycontinuation.jl: A package for homotopy continuation in
  julia.
\newblock In James~H. Davenport, Manuel Kauers, George Labahn, and Josef Urban,
  editors, {\em Mathematical Software -- ICMS 2018}, pages 458--465, Cham,
  2018. Springer International Publishing.

\bibitem{burton2014real}
Sam Burton, Cynthia Vinzant, and Yewon Youm.
\newblock A real stable extension of the {V}ámos matroid polynomial.
\newblock {\em arXiv preprint arXiv:1411.2038}, 2014.

\bibitem{MatroidsSource}
Justin Chen.
\newblock Matroids: a package for computations with matroids.
\newblock A \emph{Macaulay2} package available at
  \url{https://github.com/Macaulay2/M2/tree/master/M2/Macaulay2/packages}.

\bibitem{MatroidsArticle}
Justin Chen.
\newblock {Matroids: a \emph{Macaulay2} package}.
\newblock {\em The Journal of Software for Algebra and Geometry}, 9, 2018.

\bibitem{Choe_2004}
Young-Bin Choe, James~G Oxley, Alan~D Sokal, and David~G Wagner.
\newblock Homogeneous multivariate polynomials with the half-plane property.
\newblock {\em Advances in Applied Mathematics}, 32(1-2):88--187, 2004.

\bibitem{SumsOfSquaresSource}
Diego Cifuentes, Thomas Kahle, Pablo~A. Parrilo, and Helfried Peyrl.
\newblock {SumsOfSquares: A \emph{Macaulay2} package. Version~2.1}.
\newblock A \emph{Macaulay2} package available at
  \url{https://github.com/Macaulay2/M2/tree/master/M2/Macaulay2/packages}.

\bibitem{SumsOfSquaresArticle}
Diego Cifuentes, Thomas Kahle, Pablo~A. Parrilo, and Helfried Peyrl.
\newblock {Sums of squares in \emph{Macaulay2}}.
\newblock {\em The Journal of Software for Algebra and Geometry}, 10, 2020.

\bibitem{sage}
W.A.~Stein et~al.
\newblock Sage {M}athematics {S}oftware (version 9.2).
\newblock {S}age Development Team, 2020. \url{https://www.sagemath.org}.

\bibitem{gar}
Lars G{\.a}rding.
\newblock An inequality for hyperbolic polynomials.
\newblock {\em J. Math. Mech.}, 8:957--965, 1959.

\bibitem{Gelfand_1987}
I.~M. Gelfand, R.~M. Goresky, R.~D. MacPherson, and V.~V. Serganova.
\newblock Combinatorial geometries, convex polyhedra, and {S}chubert cells.
\newblock {\em Adv. in Math.}, 63(3):301--316, 1987.

\bibitem{Gelfand2}
I.~M. Gelfand and V.~V. Serganova.
\newblock Combinatorial geometries and the strata of a torus on homogeneous
  compact manifolds.
\newblock {\em Uspekhi Mat. Nauk}, 42(2(254)):107--134, 287, 1987.

\bibitem{M2}
Daniel~R. Grayson and Michael~E. Stillman.
\newblock Macaulay2, a software system for research in algebraic geometry.
\newblock Available at \url{http://www.math.uiuc.edu/Macaulay2/}.

\bibitem{Huh-21}
June Huh, Benjamin Schröter, and Botong Wang.
\newblock Correlation bounds for fields and matroids.
\newblock {\em Journal of the European Mathematical Society}, Jun 2021.

\bibitem{kumvam}
Mario Kummer.
\newblock A note on the hyperbolicity cone of the specialized {V}\'{a}mos
  polynomial.
\newblock {\em Acta Appl. Math.}, 144:11--15, 2016.

\bibitem{KumBez}
Mario Kummer.
\newblock Determinantal representations and {B}\'{e}zoutians.
\newblock {\em Math. Z.}, 285(1-2):445--459, 2017.

\bibitem{kummer2020spectral}
Mario Kummer.
\newblock Spectral linear matrix inequalities.
\newblock {\em Adv. Math.}, 384:Paper No. 107749, 36, 2021.

\bibitem{Kummer-2013-hyperbolic-polynomials-interlacers}
Mario Kummer, Daniel Plaumann, and Cynthia Vinzant.
\newblock Hyperbolic polynomials, interlacers, and sums of squares.
\newblock {\em Math. Program.}, 153(1, Ser. B):223--245, 2015.

\bibitem{Matsumoto-2011}
Yoshitake Matsumoto, Sonoko Moriyama, Hiroshi Imai, and David Bremner.
\newblock Matroid enumeration for incidence geometry.
\newblock {\em Discrete \& Computational Geometry}, 47(1):17–43, Nov 2011.

\bibitem{matroids9}
Dillon Mayhew and Gordon~F. Royle.
\newblock Matroids with nine elements.
\newblock {\em J. Combin. Theory Ser. B}, 98(2):415--431, 2008.

\bibitem{murota}
Kazuo Murota.
\newblock {\em Discrete convex analysis}.
\newblock SIAM Monographs on Discrete Mathematics and Applications. Society for
  Industrial and Applied Mathematics (SIAM), Philadelphia, PA, 2003.

\bibitem{ingleton}
Peter Nelson and Jorn van~der Pol.
\newblock Doubly exponentially many {I}ngleton matroids.
\newblock {\em SIAM J. Discrete Math.}, 32(2):1145--1153, 2018.

\bibitem{oxley}
James Oxley.
\newblock {\em Matroid theory}, volume~21 of {\em Oxford Graduate Texts in
  Mathematics}.
\newblock Oxford University Press, Oxford, second edition, 2011.

\bibitem{soshyperbolic}
James Saunderson.
\newblock Certifying polynomial nonnegativity via hyperbolic optimization.
\newblock {\em SIAM J. Appl. Algebra Geom.}, 3(4):661--690, 2019.

\bibitem{Vinnikov_2012}
Victor Vinnikov.
\newblock L{MI} representations of convex semialgebraic sets and determinantal
  representations of algebraic hypersurfaces: past, present, and future.
\newblock In {\em Mathematical methods in systems, optimization, and control},
  volume 222 of {\em Oper. Theory Adv. Appl.}, pages 325--349.
  Birkh\"{a}user/Springer Basel AG, Basel, 2012.

\bibitem{ineqs}
David~G. Wagner.
\newblock Matroid inequalities from electrical network theory.
\newblock {\em Electron. J. Combin.}, 11(2):Article 1, 17, 2004/06.

\bibitem{MR2134193}
David~G. Wagner.
\newblock Rank-three matroids are {R}ayleigh.
\newblock {\em Electron. J. Combin.}, 12:Note 8, 11, 2005.

\bibitem{Wagner_2009}
David~G. Wagner and Yehua Wei.
\newblock A criterion for the half-plane property.
\newblock {\em Discrete Math.}, 309(6):1385--1390, 2009.

\end{thebibliography}
 \end{document}